\documentclass{commat}

\allowdisplaybreaks

\usepackage[vcentermath,enableskew]{youngtab}
\usepackage[shortlabels]{enumitem}

\newtheorem*{thmA}{Theorem A}
\newtheorem*{thmB}{Theorem B}
\newtheorem*{thmC}{Theorem C}
\newtheorem*{thmD}{Theorem D}
\newtheorem{notation}{Notation}

\numberwithin{equation}{section}

\newcommand{\spin}{\mathbf{S}}

\newcommand{\fg}{{\mathfrak g}}
\newcommand{\fh}{{\mathfrak h}}
\newcommand{\fl}{{\mathfrak l}}
\newcommand{\fp}{{\mathfrak p}}

\newcommand{\fq}{{\mathfrak q}}
\newcommand{\ft}{{\mathfrak t}}

\newcommand{\gl}{\mathfrak{gl}}

\newcommand{\be}{\begin{equation}}
        \newcommand{\beu}{\begin{equation*}}

                \newcommand{\mult}{\mathrm{mult}}
                
                \title{%
                        Matrix formulas for multiplicities in the spin module 
                }
                
                \author{%
                        Lucas Fresse and Salah Mehdi
                }
                
                \authorinfo[%
                Lucas Fresse]{% Please, use "F. Name" format
                        Universit\'e de Lorraine, CNRS, IECL, F-54000 Nancy, France}{%
                        lucas.fresse@univ-lorraine.fr
                }
                
                \authorinfo[%
                Salah Mehdi]{% Please, use "S. Name" format
                        Universit\'e de Lorraine, CNRS, IECL, F-57000 Metz, France, and\\ New York University - Abu Dhabi, UAE}{%
                        salah.mehdi@univ-lorraine.fr, salah.mehdi@nyu.edu
                }
                
                \abstract{%
                        We obtain inductive and enumerative formulas for the multiplicities of the weights of the spin module for the Clifford algebra of a Levi subalgebra in a complex semisimple Lie algebra. Our formulas involve only matrices and tableaux, and our techniques combine linear algebra, Lie theory, and combinatorics. Moreover, this suggests a relationship with complex nilpotent orbits. The case of the special linear Lie algebra $\mathfrak{sl}(n,{\mathbb C})$ is emphasized.
                }
                
                

                \VOLUME{32}
                \YEAR{2024}
                \NUMBER{1}
                \firstpage{203}
                \DOI{https://doi.org/10.46298/cm.11156}
                
\begin{document}

                        \section{Introduction}\label{intro}
                        Spinors were used by the Nobel laureate Paul Dirac in the late 1920's to describe the relativistic quantum kinematics of a~free electron on the Minkowski spacetime~\cite{D}. One feature of the Dirac equation is that it involves a~differential operator, with coefficients in a~Clifford algebra, whose square differs from the Laplacian by a~constant. This Clifford algebra admits a~unique irreducible finite-dimensional representation, up to isomorphism, known as the spin module. The solutions to the Dirac equation are precisely sections of a~bundle over the Minkowski spacetime induced by the spin module.  Dirac operators and spinors have been used extensively in a~wide spectrum of areas such as, just to mention a~few, Seiberg-Witten invariants~\cite{SWa},~\cite{SWb}, non-commutative geometry~\cite{C},~\cite{CMa}, differential geometry~\cite{Fr} or representation theory of Lie groups~\cite{HP}.
                        
                        The aim of this paper is to provide elementary and easily implementable formulas for the multiplicities of the spin module for matrix Lie algebras. Our formulas involve matrices with integer entries and weighted tableaux.
                        Our main results are outlined in Theorems A, B, C, and D below. Before describing our results, for the convenience of the reader, we start with a~brief review on multiplicities for finite-dimensional modules of reductive Lie algebras.
                        
                        \subsection{Multiplicity formulas for highest weight modules}\label{section-1.1}
                        Finite-dimensional modules of complex semisimple Lie algebras are completely characterized by weights and their multiplicities: two modules with the same weights must be isomorphic. In other words, two irreducible finite-dimensional modules are isomorphic if, and only if, they have the same character. The character encodes the multiplicities of a~module. To compute multiplicities, numerous formulas of different nature exist, each of which has its own strengths and weaknesses.
                        
                        More precisely, let ${\mathfrak h}$ be a~complex semisimple Lie algebra, ${\mathfrak t}\subset{\mathfrak h}$ a~Cartan subalgebra, $\Phi({\mathfrak h})$ the set of ${\mathfrak t}$-roots in ${\mathfrak h}$ and $W({\mathfrak h},{\mathfrak t})$ the Weyl group generated by the root reflections. Let ${\mathfrak t}_{\mathbb R}\subset{\mathfrak t}$ be the real vector subspace spanned by the co-root vectors, so that ${\mathfrak t}$ is the complexification of ${\mathfrak t}_{\mathbb R}$. Since the restriction to ${\mathfrak t}$ of the Killing form of ${\mathfrak g}$ remains non-degenerate, there exists a~non-degenerate complex bilinear form $\langle\;,\;\rangle$ on ${\mathfrak t}^*$. In particular, $\langle\;,\;\rangle$ is positive definite on ${\mathfrak t}_{\mathbb R}$. This form will be used to identify ${\mathfrak t}_{\mathbb R}$ and ${\mathfrak t}_{\mathbb R}^*$. Write $\mid\mid\;\mid\mid$ for the induced norm on ${\mathfrak t}_{\mathbb R}^*$. Fixing an order on ${\mathfrak t}_{\mathbb R}^*$ defines a~set $\Phi^+ ({\mathfrak h})$ of positive roots and a~set $\Delta({\mathfrak h})$ of simple roots. 
                        
                        By the highest weight theorem, the set of finite dimensional irreducible ${\mathfrak h}$-modules is in bijective correspondence with the set of $\Phi^+ ({\mathfrak h})$-dominant algebraically integral weights. Let $V[\lambda]$ be a~finite-dimensional irreducible ${\mathfrak h}$-module with highest weight $\lambda\in{\mathfrak t}_{\mathbb R}^*$, write $\Lambda(V[\lambda])\subset{\mathfrak t}_{\mathbb R}^*$ for the set of ${\mathfrak t}$-weights of $V[\lambda]$. In fact, any dominant weight of $V[\lambda]$ is of the form $\lambda-\sum n_i\alpha_i$, where $\alpha_i\in\Delta({\mathfrak h})$ and $n_i$ are non-negative integers. If $\text{mult}_{V[\lambda]}(\mu)$ denotes the multiplicity of the weights $\mu$ in $V[\lambda]$, i.e., the dimension of the $\mu$-weight space $V[\lambda]_\mu$, then one has $\text{mult}_{V[\lambda]}(\lambda)=1$. Moreover, the weights and their multiplicities are invariant under the Weyl group $W({\mathfrak h},{\mathfrak t})$, and it is therefore sufficient to determine the multiplicities of the dominant weights. The character of $V[\lambda]$ can be defined as the following element of the group ring ${\mathbb Z}[{\mathfrak t}_{\mathbb R}^* ]$:
                        \begin{equation*}
                                \chi(V[\lambda])=\sum_{\mu\in\Lambda(V[\lambda])}\text{mult}_{V[\lambda]}(\mu)e^\mu.
                        \end{equation*}
                        This definition is compatible with the direct sum and the tensor product of modules. The Weyl character formula states that
                        \begin{equation}\label{weyl}
                                \chi(V[\lambda])\sum_{w\in W({\mathfrak h},{\mathfrak t})}(-1)^{\ell(w)}e^{w(\rho({\mathfrak h}))}=\sum_{w\in W({\mathfrak h},{\mathfrak t})}(-1)^{\ell(w)}e^{w(\lambda+\rho({\mathfrak h}))}
                        \end{equation}
                        where $\rho({\mathfrak h})=\frac{1}{2}\sum_{\alpha\in\Phi^+ ({\mathfrak h})}\alpha$ is half the sum of positive roots and $\ell(w)$ is the length of $w$. In particular, one deduces the Weyl dimension formula:
                        \begin{equation}\label{weyldim}
                                \text{dim}(V[\lambda])=\prod_{\alpha\in\Phi^+ ({\mathfrak h})}\frac{\langle \lambda+\rho({\mathfrak h}),\alpha\rangle}{\langle\rho({\mathfrak h}),\alpha\rangle}.
                        \end{equation}
                        For a~thorough discussion on finite-dimensional modules for complex semisimple Lie algebras, see for instance~\cite{GW} or~\cite{H}.
                        
                        One of the classical multiplicity formulas is the Freudenthal formula which can be stated as follows~\cite{F}:
                        \begin{equation}\label{freud}
                                \big(2\langle\lambda-\mu,\mu+\rho({\mathfrak h})\rangle+\mid\mid\lambda-\mu\mid\mid^2 \big)\text{mult}_{V[\lambda]}(\mu)=2\sum_{\alpha\in\Phi^+ ({\mathfrak h})}
                                \sum_{k\geq 1}\langle\mu+k\alpha,\alpha\rangle\text{mult}_{V[\lambda]}(\mu+k\alpha).
                        \end{equation}
                        Note that if $\mu$ is dominant then $\langle\lambda-\mu,\mu+\rho({\mathfrak h})\rangle$ must be positive. Though the Freudenthal multiplicity formula is a~simple iterative formula, it is not easy to implement. Indeed, one needs to know all the higher weights $\mu+k\alpha$ that are involved, along with their multiplicities.
                        
                        Another classical multiplicity formula is due to Kostant. It involves a~partition function ${\mathcal P}:{\mathfrak t}_{\mathbb R}^* \rightarrow{\mathbb N}$ defined as follows: ${\mathcal P}(0)=1$ and ${\mathcal P}(\mu)$ is the number of different ways to write $\mu$ as a~sum of positive roots, so that ${\mathcal P}(\mu)=0$ unless $\mu$ is positive. Now, Kostant multiplicity formula asserts that~\cite{K}:
                        \begin{equation}\label{kostant}
                                \text{mult}_{V[\lambda]}(\mu)=\sum_{w\in W({\mathfrak h},{\mathfrak t})}(-1)^{\ell(w)}{\mathcal P}(w(\lambda+\rho({\mathfrak h}))-(\mu+\rho({\mathfrak h}))).
                        \end{equation}
                        Although this formula computes directly the multiplicity of a~weight, the summation goes over the full Weyl group $W({\mathfrak h},{\mathfrak t})$ and requires keeping track of those summands with a~nonzero contribution. Moreover, Kostant's formula involves a~counting function whose values are not always simple to determine.
                        
                        There are several more involved multiplicity formulas such as, just to mention a~few, Lusztig's formula based on intersection theory of Deligne-Goresky-MacPherson and singularities of Schubert varieties~\cite{Lus}, Littelmann's formula involving paths and root operators~\cite{L} and Sahi's formula which expresses multiplicities as sums of rational numbers related to the dual affine Weyl group~\cite{S}.
                        
                        In this paper, we prove explicit formulas for multiplicities of the spin module when ${\mathfrak h}$ is a~quadratic subalgebra of a~semisimple Lie algebra $\fg$. Our interest in this setting is motivated by both Dirac operators and nilpotent orbits.
                        On the one hand, in our first paper~\cite{FM} we studied the approximation of nilpotent orbits for simple Lie groups.
                        On the other hand, Dirac operators and nilpotent orbits play, each on its own, an important role in representation theory. For instance, Springer correspondence establishes a~bijection between the set of complex nilpotent orbits of a~semisimple algebraic group and the set of (equivalences of) irreducible representations of the (full) Weyl group (cf.~\cite{Spr}), while discrete or principal series representations of a~semisimple Lie group can be explicitly realized as spaces of $L^2$ or smooth sections of suitable twists of the spin bundle over homogeneous spaces, these sections being harmonic for the corresponding Dirac operator (cf.~\cite{AS},~\cite{P},~\cite{MZ2}).
                        
                        Moreover, nilpotent orbits and Dirac operators can be related in a~precise way through associated cycles of Harish-Chandra modules and the asymptotics of global characters (cf.~\cite{MPVZ}). The present paper provides combinatorial formulas for the multiplicities of the spin module in terms of matrices and tableaux. Our formulas suggest a~relationship between multiplicities and nilpotent orbits.
                        
                        \subsection{Construction of the spin module over a~Levi subalgebra}\label{section-1.2}
                        
                        The general setting of our paper is as follows : ${\mathfrak h}$ is a~Levi subalgebra of a~complex semisimple Lie algebra
                        ${\mathfrak g}$ containing a~Cartan subalgebra ${\mathfrak t}$ of ${\mathfrak g}$.
                        It is worth to mention that $\mathfrak{h}$ need not be semisimple but it is reductive and highest weight theory still applies in the following way.
                        Write ${\mathfrak h}=\lbrack{\mathfrak h},{\mathfrak h}\rbrack\oplus{\mathcal Z}_{\mathfrak h}$, where ${\mathcal Z}_{\mathfrak h}$ is the center of ${\mathfrak h}$ and ${\mathfrak t}={\mathfrak t}^\prime\oplus{\mathcal Z}_{\mathfrak h}$ where ${\mathfrak t}^\prime$ is a~Cartan subalgebra of the semisimple part $\lbrack{\mathfrak h},{\mathfrak h}\rbrack$ of ${\mathfrak h}$. If $V$ is an irreducible ${\mathfrak h}$-module then, by Schur's lemma, ${\mathcal Z}_{\mathfrak h}$ acts by scalars on $V$. In particular, $V$ is an irreducible highest weight module for $\lbrack{\mathfrak h},{\mathfrak h}\rbrack$, say with highest weight $\lambda^\prime$. We will consider $V$ as a~highest weight ${\mathfrak h}$-module with highest weight $\lambda$ where $\lambda_{\mid_{\lbrack{\mathfrak h},{\mathfrak h}\rbrack}}=\lambda^\prime$ and $\lambda_{\mid_{{\mathcal Z}_{\mathfrak h}}}$ is the scalar by which ${\mathcal Z}_{\mathfrak h}$ acts on $V$. (See~\cite{Di} for more details.)
                        
                        If ${\mathfrak p}^\pm$ are the two standard parabolic subalgebras having ${\mathfrak h}$ as a~Levi factor with nilradicals ${\mathfrak u}^\pm$, then the vector space
                        \begin{equation*}
                                {\mathfrak q}:={\mathfrak u}^+ \oplus {\mathfrak u}^-
                        \end{equation*}
                        is the orthogonal complement of ${\mathfrak h}$ in ${\mathfrak g}$, with respect to the Killing form. Note that ${\mathfrak u}^+$ and ${\mathfrak u}^-$ are maximally dual isotropic in ${\mathfrak q}$. Let $\Phi$, $\Phi({\mathfrak h})$ and $\Phi({\mathfrak q})$ be the sets of ${\mathfrak t}$-roots in ${\mathfrak g}$, ${\mathfrak h}$ and ${\mathfrak q}$ respectively. Fix compatible positive systems $\Phi^+$, $\Phi^+ ({\mathfrak h})$ and $\Phi^+ ({\mathfrak q})$. As above, $W({\mathfrak g},{\mathfrak t})$ will stand for the Weyl group of ${\mathfrak t}$-roots in ${\mathfrak g}$, while $\rho({\mathfrak g})$, $\rho({\mathfrak h})$ and $\rho({\mathfrak q})$ for half the sums of positive ${\mathfrak t}$-roots in ${\mathfrak g}$, ${\mathfrak h}$ and ${\mathfrak q}$ respectively. If $\textbf{C}({\mathfrak q})$ denotes the Clifford algebra of ${\mathfrak q}$ and
                        \begin{equation*}
                                \textbf{S}:=\bigwedge{\mathfrak u}^+,
                        \end{equation*}
                        then the spin module for ${\mathfrak h}$ is defined by the composition map:
                        \begin{equation*}
                                {\mathfrak h}\xrightarrow{\text{ad}}{\mathfrak s}{\mathfrak o}({\mathfrak q})\subset{\bf C}({\mathfrak q})\xrightarrow{\gamma}\text{End}(\textbf{S})
                        \end{equation*}
                        where $\gamma$ is the Clifford multiplication. The weights of $\spin$ form the set
                        \begin{equation}\label{spinweight}
                                \Lambda(\textbf{S})=\Big\{\mu(A):=\frac{1}{2}\Big(\sum_{\alpha\not\in A}\alpha-\sum_{\alpha\in A}\alpha\Big)\;\mid\;A\subset\Phi^+ ({\mathfrak q})\Big\}\subset\mathfrak{t}_{\mathbb R}^*,
                        \end{equation}
                        so that the multiplicity of $\beta\in\Lambda(\textbf{S})$ is given by
                        \begin{equation*}
                                \text{mult}_{\spin}(\beta)=\#\big\{A\subset\Phi^+ ({\mathfrak q})\mid\mu(A)=\beta\big\}.
                        \end{equation*}
                        %
                        %%%%%%%%%%%%
                        See Example~\hyperref[E3.4b]{2.3(b)} for an explicit description of $\spin$ for ${\mathfrak s}{\mathfrak l}(n)$ and a~particular Levi subalgebra. It should be noted that even in such a~case, the action of ${\mathfrak h}$ on $\spin$ is not the standard ``matrix" action but a~rather different action involving the Clifford multiplication. This is a~special case of the spin representation defined for pairs $({\mathfrak g},{\mathfrak h})$, where ${\mathfrak h}$ is a~reductive subalgebra of ${\mathfrak g}$ (see, e.g.,~\cite[Chapter 6]{GW} or~\cite[Section 1]{MZ}). For the purpose of this paper, we will focus on the case where $\fh$ is a~Levi subalgebra of $\fg$.
                        
                        Unlike $W({\mathfrak h},{\mathfrak t})$, the Weyl group $W({\mathfrak g},{\mathfrak t})$ does not act on spin weights. In Section~\ref{S3.1}, we observe that $\Lambda(\textbf{S})$ is in fact stable under the subgroup
                        \begin{equation*}
                                W^\prime:=\{w\in W({\mathfrak g},{\mathfrak t})\;\mid\;w(\Phi({\mathfrak q}))=\Phi({\mathfrak q})\}.
                        \end{equation*}
                        In particular, to compute the weights of the spin module it is enough to pick one Levi subalgebra in each $W({\mathfrak g},{\mathfrak t})$-orbit on the set of Levi subalgebras in ${\mathfrak g}$ and then pick one weight in each $W^\prime$-orbit on $\Lambda(\textbf{S})$.
                        
                        Recall that standard parabolic subalgebras of ${\mathfrak g}$ are parametrized as ${\mathfrak p}_I$, with Levi factor ${\mathfrak l}_I$, where $I$ is a~subset of the set $\Delta$ of simple ${\mathfrak t}$-roots in ${\mathfrak g}$. Let $\textbf{S}_I:=\bigwedge ({\mathfrak l}_I\cap{\mathfrak u}^+)$ be the spin module associated with the pair $({\mathfrak l}_I,{\mathfrak l}_I\cap{\mathfrak l})$. In the case when the subset $A\subset\Phi^+ ({\mathfrak q})$ can be decomposed as the union $A=\bigcup_j A_j$
                        of subsets $A_j\subset\Phi^+ \cap\langle I_j\rangle$ (hereafter $\langle \cdot \rangle$ stands for the linear span), where $I_j\subset\Delta$ are such that $I_i\perp I_j$ whenever $i\neq j$, we obtain the following formula for
                        multiplicities (Proposition~\ref{L3.8-new}):
                        \begin{equation*}
                                \text{mult}_{\spin}(\mu(A))=\prod_j \text{mult}_{\spin_{I_j}}(\mu({A_j})).
                        \end{equation*}
                        In the special case where ${\mathfrak h}$ coincides with the Cartan subalgebra ${\mathfrak t}$, we see that the multiplicity of $\mu(A)$ is equal to $1$
                        exactly when $A$ is saturated, i.e., for all $\alpha,\beta\in A$ with $\alpha+\beta\in\Phi$, then $\alpha+\beta$ must be in $A$; and for all
                        $\gamma,\delta\in \Phi^+ \setminus A$, with $\gamma+\delta\in\Phi$ then one has $\gamma+\delta\notin A$ (Proposition~\ref{pro-saturated}).
                        
                        \subsection{Main results involving inductive formulas for multiplicities of spin modules over Levi subalgebras}
                        In Section~\ref{section-inductive}, we prove an inductive formula for the multiplicities of the spin module, when ${\mathfrak g}={\mathfrak s}{\mathfrak l}(n,{\mathbb C})$ and $n\geq 1$. Here, we choose the usual positive system 
                        \[\Phi^+ =\{\epsilon_i-\epsilon_j\mid 1\leq i<j\leq n\}.\] 
                        Recall from~\cite[pp.\,112-113]{CM} that Levi subalgebras in ${\mathfrak s}{\mathfrak l}(n,{\mathbb C})$ (up to $W(\mathfrak{g},\mathfrak{t})$ action) are determined by nondecreasing sequences $c=(c_1,\ldots,c_k)$ summing to $n$, that is,
                        \begin{equation*}
                                c_1\leq c_2\leq\cdots\leq c_k\;\text{ and }\;c_1+c_2+\cdots+c_k=n.
                        \end{equation*}
                        Fixing such a~sequence $c$, the corresponding Levi subalgebra $\mathfrak{h}$ consists of blockwise diagonal matrices with blocks of sizes $c_1,\ldots, c_k$ respectively (see also Example~\ref{E3.4}).
                        
                        If $\mu=\sum_{i=1}^n b_i\epsilon_i$ is a~weight of the spin module, we will write
                        \[
                        \text{mult}_{\spin}(\mu):=\text{mult}_{(c_1,\ldots,c_k)}(b_1,\ldots, b_n).
                        \]
                        Note that one must have $\sum_i b_i=0$. If $X$ is a~set, ${\mathcal P}_k(X)$ will denote the set of subsets $J\subset X$ with $k$ elements and $\mathbf{1}_J:X\to \{0,1\}$ the function given by $\mathbf{1}_J(x)=1$ or $0$ depending on whether $x\in J$ or $x\notin J$.
                        
                        \begin{thmA}[see Theorem~\ref{T-inductive-Levi}]
                                When $\fh$ is a~Levi subalgebra of $\fg=\mathfrak{sl}(n,\mathbb{C})$, with the above notation,~we~have \rm
                                \begin{eqnarray*}
                                        &&\text{mult}_{(c_1,\ldots,c_k)}(b_1,\ldots,b_n) \\
                                        &&=\sum_{J}\text{mult}_{(c_1-1,c_2,\ldots,c_k)}\Big(b_2,\ldots,b_{c_1},b_{c_1+1}+\mathbf{1}_J(c_1+1)-\frac{1}{2},\ldots,b_n+\mathbf{1}_J(n)-\frac{1}{2}\Big),
                                \end{eqnarray*}
                                where $J$ runs over $\mathcal{P}_{b_1+\frac{n-c_1}{2}}(\{c_1+1,\ldots,n\})$ in the sum.
                        \end{thmA}
                        
                        In Section~\ref{section-Cartan-sln}, we continue with ${\mathfrak g}={\mathfrak s}{\mathfrak l}(n,{\mathbb C})$ but we now assume that the Levi algebra ${\mathfrak h}$ coincides with the Cartan subalgebra ${\mathfrak t}$. Let $\text{Part}(\binom{n}{2};n)$ be the set of partitions of $\binom{n}{2}$ with at most $n$ parts, i.e., sequences of positive integers
                        $\lambda=(\lambda_1\geq\ldots\geq\lambda_n)$ with $\lambda_1+\ldots+\lambda_n=\binom{n}{2}$.
                        Write $\mathcal{P}(n)$ for the subset of partitions $\lambda\in\mathrm{Part}(\binom{n}{2};n)$ such that $\lambda\preceq(n-1,n-2,\ldots,1)$, where $\preceq$ stands for the dominance order on partitions (see Section~\ref{section-5.1}).
                        Define
                        \begin{eqnarray*}
                                \mu[\lambda]:=(\lambda_1-\lambda_2)\varpi_1+\ldots+(\lambda_{n-1}-\lambda_n)\varpi_{n-1}
                        \end{eqnarray*}
                        where $\varpi_i=\epsilon_1+\ldots+\epsilon_i$ denotes the $i$th fundamental weight ($i=1,\ldots,n-1$). One can check that the map
                        $\lambda\mapsto\mu[\lambda]$ is a~bijective correspondence between the set ${\mathcal P}(n)$ and the set $\Lambda^+ (\textbf{S})$ of dominant weights in $\Lambda(\textbf{S})$ (Proposition~\ref{P4.4}). We view $\lambda\in\text{Part}(\binom{n}{2};n)$ as a~Young diagram. For $p\in\{1,\ldots,n\}$, we call $p$-marking a~subset $\beta\subset\lambda$ of boxes such that: $\beta$ contains exactly $n-1$ boxes; $\beta$ contains all the boxes of the $p$-th row of $\lambda$ and in the other rows of $\lambda$, only the rightmost box may belong to $\beta$ (Notation~\ref{notation-markings}). We denote by $M_p(\lambda)$ the set of $p$-markings of $\lambda$.
                        
                        \begin{thmB}[see Theorem~\ref{P2}]
                                When $\fh=\ft$ is a~Cartan subalgebra of $\fg=\mathfrak{sl}(n,\mathbb{C})$, with the above notation, we have \rm
                                \begin{eqnarray*}
                                        \text{mult}_{\spin}(\mu[\lambda])=\sum_{\beta\in M_p(\lambda)}\text{mult}_{\spin}(\mu[\lambda-\beta]).
                                \end{eqnarray*}
                        \end{thmB}
                        
                        We point out two consequences of Theorem B. Suppose that $\nu\in\text{Part}(\binom{m}{2};m)$ and $\pi\in~\text{Part}(\binom{p}{2};p)$ with $\pi_1\leq p+\nu_m$, so that 
                        \[\lambda:=(p+\nu_1,\ldots,p+\nu_m,\pi_1,\ldots,\pi_p)\in\text{Part}\left(\binom{m+p}{2};m+p\right).\]
                        Then we obtain in Proposition~\ref{P-5.13} that $\lambda\in{\mathcal P}(n)$ if, and only if, $\nu\in{\mathcal P}(m)$ and $\pi\in{\mathcal P}(p)$, with
                        \begin{equation*}
                                \text{mult}_{\spin}(\mu[\lambda])=\text{mult}_{\spin}(\mu[\nu])\times\text{mult}_{\spin}(\mu[\pi]).
                        \end{equation*}
                        Another consequence arises when the weight $\mu=\rho(\mathfrak{g})-\alpha$ is a~shift of $\rho(\mathfrak{g})$ by a~single positive root $\alpha\in\Phi^+ ({\mathfrak q})$. In this case, Proposition~\ref{P-5.15} establishes that the multiplicity of $\mu$ is a~specific power of $2$.
                        
                        \subsection{Main results involving enumerative formulas through suitable sets of tableaux}\label{section-1.4}
                        We obtain an enumerative formula for multiplicities in the following two opposite extreme cases in $\fg=\mathfrak{sl}(n,\mathbb{C})$:
                        \begin{enumerate}[(a)]
                                \item\label{sln_a} the case of a~Cartan subalgebra $\mathfrak{h}=\mathfrak{t}$,
                                \item\label{sln_b} the case where $\fh$ is the Levi subalgebra of a~maximal parabolic subalgebra $\fp$.
                        \end{enumerate}
                        In case \ref{sln_a}, we introduce the notion of spin tableau. 
                        
                        For $n\geq 1$ and $\lambda=(\lambda_1,\ldots,\lambda_n)\in{\mathcal P}(n)$, a~spin tableau $\tau$ of shape $\lambda$ is a~tableau satisfying the following conditions (Definition~\ref{D-T}):
                        \begin{itemize}
                                \item for every $i\in\{1,\ldots,n-1\}$, $\tau$ contains $i$ boxes of entry $i$, all located within the first $i+1$ rows,
                                \item the rows (resp. columns) of $\tau$ are nondecreasing from left to right (resp. top to bottom), moreover on the $i$-th row, the entries which are at least equal to $i$ are increasing.
                        \end{itemize}
                        For example, there are two spin tableaux of shape $\lambda=(4,2,2,1,1)$, namely
                        \[
                        \young(1234,23,34,4,4)\quad\mbox{and}\quad\young(1234,24,34,3,4).
                        \]
                        To every spin tableau $\tau$, we attach an explicit integer $N_\tau$ (see Definition~\hyperref[D-Tb]{4.9(b)}).
                        
                        \begin{thmC}[see Theorem~\ref{theorem-tableaux}]
                                When $\fh=\ft$ is a~Cartan subalgebra of $\fg=\mathfrak{sl}(n,\mathbb{C})$, we have
                                \begin{equation*}
                                        \text{\rm mult}_{\spin}(\mu[\lambda])=\sum_{\tau\in\mathcal{ST}(\lambda)}N_\tau
                                \end{equation*}
                                where $\mathcal{ST}(\lambda)$ denotes the set of spin tableaux of shape $\lambda$.
                        \end{thmC}
                        
                        In Section~\ref{section6}, we consider the case \ref{sln_b} where ${\mathfrak h}$ is the Levi factor of a~maximal parabolic subalgebra in ${\mathfrak g}={\mathfrak s}{\mathfrak l}(n,{\mathbb C})$. In other words, me may assume that
                        \[
                        \mathfrak{h}=\left\{
                        \begin{pmatrix}
                                x & 0 \\ 0 & y
                        \end{pmatrix}
                        :x\in\gl(p,\mathbb{C}),\ y\in\gl(q,\mathbb{C})\right\}
                        \]
                        with $p+q=n$. In particular, every weight in $\Lambda(\textbf{S})$ can be obtained from a~weight in
                        \[
                        \Lambda^\prime(\textbf{S}):=\{\mu=(\mu_1,\ldots,\mu_p,\nu_1,\ldots,\nu_q)\in\Lambda(\textbf{S})\; :\; \mu_1\geq\cdots\geq\mu_p,\;\;\nu_1\geq\cdots\geq\nu_q\}
                        \]
                        under the action of $W({\mathfrak h},{\mathfrak t})=\mathfrak{S}_p\times\mathfrak{S}_q$ (see Remark~\ref{lambdaprime}).
                        Then we show that the set $\Lambda'(\spin)$ is in bijection with the following set of pairs of partitions
                        \begin{gather*}
                                \mathcal{P}':=\{(\alpha,\beta):\alpha,\beta \text{ partitions of } m\leq n \text{ with at most } q \text{, resp. } p \text{ parts,} \\* 
                                \text{ and such that } \beta\preceq{}^t \alpha \}
                        \end{gather*}
                        where again $\preceq$ stands for the dominance order on partitions
                        (see Corollary~\ref{C5.3}).
                        More precisely, the bijection is given by
                        \begin{equation*}
                                (\alpha,\beta)\mapsto\mu(\alpha,\beta):=\left(\alpha_1-\frac{q}{2},\ldots,\alpha_p-\frac{q}{2},\frac{p}{2}-\beta_q,\ldots,\frac{p}{2}-\beta_1\right).
                        \end{equation*}
                        We then express the multiplicity of the weight $\mu(\alpha,\beta)$ in terms of tableaux as follows. Following Definition~\ref{RTtableau}, a~row-tableau of shape $\alpha$ and weight $\beta$ is a~numbering of the boxes of $\alpha$ (viewed as a~Young diagram) such that the numbering comprises $\beta_i$ occurrences of $i$ for all integer $i$, and such that the entries increase from left to right along the rows.
                        
                        \begin{thmD}[see Theorem~\ref{T2-new}]
                                For $(\alpha,\beta)\in\mathcal{P}'$ as above, we have
                                \begin{equation*}
                                        \mult_{\spin}(\mu(\alpha,\beta))=\# RT(\alpha,\beta),
                                \end{equation*}
                                where $RT(\alpha,\beta)$ denotes the set of row-tableaux of shape $\alpha$ and weight $\beta$.
                        \end{thmD}
                        
                        \noindent
                        For instance, given $p=q=3$ and $n=6$, we have
                        \[
                        \mult_{\spin}(\mu((3,2,1),(2,2,2)))=\#\left\{\young(123,12,3),\ \young(123,13,2),\ \young(123,23,1)\,\right\}=3.
                        \]

                        Our general approach for obtaining Theorems A, B, C, and D relies on computational linear algebra, namely encoding weights of the spin module with suitable sets of matrices with coefficients in $\{0,1\}$. For instance, for establishing Theorem D, we first
                        show that
                        \[
                        \text{mult}_{\spin}(\mu(\alpha,\beta))=\#\mathcal{M}_{p,q}(\alpha,\beta)
                        \]
                        where
                        $\mathcal{M}_{p,q}(\alpha,\beta)$ is the set of matrices $a=(a_{i,j})$ of size $(p,q)$ and coefficients in $\{0,1\}$ such that the sums of the coefficients in the $i$-th row (resp. $j$-th column) coincides with the $i$-th part of the partition $\alpha$ (resp. the $j$-th part of $\beta$). Finally, we define a~bijection
                        \[
                        \mathcal{M}_{p,q}(\alpha,\beta)\cong RT(\alpha,\beta).
                        \]
                        We think that the combinatorics of the set $\mathcal{M}_{p,q}(\alpha,\beta)$ is by itself an interesting problem.
                        
                        As a~final point, we mention that the combinatorial formulas that we have obtained involve partitions and sets of tableaux. Since it is known that nilpotent orbits of classical Lie algebras can be parametrized with certain partitions (see~\cite{CM}), our formulas suggest a~possible geometric interpretation of multiplicities of the spin module in terms of nilpotent orbits. This will be the aim of a~prospective project.
                        
                        \subsection*{Notation} In the sequel, we keep the notation of Sections~\ref{section-1.2}--\ref{section-1.4}. The base field is $\mathbb{C}$ and by $\mathfrak{gl}(n)$, $\mathfrak{sl}(n)$,\dots we denote the classical Lie algebras $\mathfrak{gl}(n,\mathbb{C})$, $\mathfrak{sl}(n,\mathbb{C})$,\dots
                        
                        For a~set $\mathbb{K}$, we denote by $\mathcal{M}_{p,q}(\mathbb{K})$ the set matrices of size $(p,q)$ and coefficients in $\mathbb{K}$; we write for simplicity $\mathcal{M}_p(\mathbb{K})=\mathcal{M}_{p,p}(\mathbb{K})$.
                        By $\mathfrak{S}_n$ we denote the symmetric group of permutations of $\{1,\ldots,n\}$. By a~partition of $n$ we mean a~nonincreasing sequence of positive integers $\lambda=(\lambda_1\geq\cdots\geq\lambda_k)$ such that $\lambda_1+\cdots+\lambda_k=n$; then we write $\lambda\vdash n$. By $\#A$ we denote the cardinal of a~set $A$. Other notation will be introduced in the text in the due places.
                        An index of notation can be found at the end of the paper.

                        \section{Weights of the spin module in the case of a~Levi subalgebra}
                        
                        We use the notation of Section~\ref{section-1.2}. In particular, we suppose that $\mathfrak{h}$ is a~Levi subalgebra of $\mathfrak{g}$ that contains $\mathfrak{t}$, and we consider the spin module $\mathbf{S}$ associated to $\mathfrak{h}$ and its set of weights $\Lambda(\mathbf{S})$ (see \eqref{spinweight}).
                        In this section, we give general properties of weights which will be required in the sequel. In general, the action of the Weyl group $W=W(\mathfrak{g},\mathfrak{t})$ does not restrict to $\Lambda(\spin)$ and, in Section~\ref{section-2.1}, we point out a~subgroup $W'\subset W$
                        that does act on $\Lambda(\spin)$ by preserving the multiplicities.
                        In Section~\ref{section-2.2}, we describe inductive properties of multiplicities relative to Levi subalgebras $\fl\subset\fh$. In Section~\ref{section-2.3}, we give general properties in the case where $\fh=\ft$; for instance, we characterize the weights of multiplicity one.
                        
                        \subsection{Action of Weyl group elements}\label{S3.1}\label{section-2.1}
                        The Weyl group $W$ acts by conjugation on the standard Levi subalgebras of $\mathfrak{g}$ and this induces an action on the collection of corresponding spin modules. Specifically,
                        given $w\in\nolinebreak W$, we can consider the Levi subalgebra ${}^w \fh:=w\fh w^{-1}$
                        and the decomposition $\fg=\nolinebreak{}^w \fh\oplus\nolinebreak{}^w \fq$ where
                        \[
                        {}^w \fq=w\fq w^{-1}=\bigoplus_{\alpha\in w(\Phi(\fq))}\fg_\alpha\quad\mbox{and}\quad {}^w \fh=\ft\oplus\bigoplus_{\alpha\in w(\Phi(\fh))}\fg_\alpha.
                        \]
                        We denote by ${}^w \mathbf{S}$ the spin module relative to ${}^w \fh$ and ${}^w \fq$ in the sense of Section~\ref{section-1.2}.
                        
                        Moreover, we define
                        \[
                        W'=\{w\in W:w(\Phi(\mathfrak{q}))=\Phi(\mathfrak{q})\},
                        \]
                        which is a~subgroup of $W$, containing $W(\mathfrak{h},\mathfrak{t})$, and such that
                        \[
                        w\in W'\quad\Rightarrow\quad {}^w \fh=\fh,\ {}^w \fq=\fq,\mbox{ and }{}^w \spin=\spin.
                        \]
                        
                        \begin{example}\label{E1}
                                
                                \begin{enumerate}[(a)]
                                        \item\label{E1a} If $\mathfrak{h}=\mathfrak{t}$, then $\Phi(\fq)=\Phi$ and $W'=W$.
                                        \item\label{E1b} If $\mathfrak{h}=\mathfrak{gl}(p)\times\mathfrak{gl}(q)\subset\mathfrak{gl}(p+q)=\mathfrak{g}$, then $\Phi(\mathfrak{q})=\{\pm(\epsilon_i-\epsilon_j):1\leq i\leq p<j\leq p+q\}$ and $W(\mathfrak{h},\mathfrak{t})\cong\mathfrak{S}_p\times\mathfrak{S}_q$.
                                        We have
                                        \begin{eqnarray*}
                                                W' & = & \{w\in\mathfrak{S}_{p+q}:w(\{1,\ldots,p\})=\{1,\ldots,p\}\ \mbox{or}\ \{p+1,\ldots,p+q\}\} \\
                                                & = & \left\{
                                                \begin{array}
                                                        {ll}
                                                        W(\fh,\ft) & \mbox{if $p\not=q$}, \\
                                                        W(\mathfrak{h},\ft)\rtimes\langle\sigma\rangle & \mbox{if $p=q$},
                                                \end{array}
                                                \right.
                                        \end{eqnarray*}
                                        where $\sigma=
                                        \begin{pmatrix}
                                                1 & \cdots & p & p+1 & \cdots & 2p \\ p+1 & \cdots & 2p & 1 & \cdots & p
                                        \end{pmatrix}
                                        $.
                                \end{enumerate}

                        \end{example}
                        
                        \begin{lemma}\label{L-w-Levi}
                                \begin{enumerate}[(a)]
                                        \item For every $w\in W$, we have $w(\Lambda(\mathbf{S}))=\Lambda({}^w \mathbf{S})$. Moreover, for all $\mu\in\Lambda(\mathbf{S})$, the weights $\mu$ (for $\mathbf{S}$) and $\mu':=w(\mu)$ (for ${}^w \mathbf{S}$) have the same multiplicity.
                                        \item In particular, the set $\Lambda(\mathbf{S})$ is stable under $W'$ and, moreover, for every $\mu\in\Lambda(\mathbf{S})$, $w\in W'$, we have $\mult_\spin(w(\mu))=\mult_\spin(\mu)$.
                                        \item For every $\mu\in\Lambda(\mathbf{S})$, we have $-\mu\in\Lambda(\mathbf{S})$, and $\mult_\spin(-\mu)=\mult_\spin(\mu)$.
                                \end{enumerate}
                        \end{lemma}
                        
                        \begin{proof}
                                (a) Let $\mu=\mu(A)$ for a~subset $A\subset\Phi^+ (\fq)$. Fix $w\in W$ and set for simplicity $\fq'={}^w \fq$. Let
                                \[
                                w(A):=\{\alpha\in\Phi^+ (\fq'):w^{-1}(\alpha)\in A\quad\mbox{or}\quad -w^{-1}(\alpha)\in \Phi^+ (\fq)\setminus A\},
                                \]
                                and additionally, denote 
                                \begin{align*}
                                        &\Phi^{+}_\setminus=\left\{\alpha\in\Phi^+(\fq'):w^{-1}(\alpha)\in\Phi^+(\fq)\setminus A\right\}, \quad 
                                        \Phi^{+}_\cap=\left\{\alpha\in\Phi^+(\fq'):w^{-1}(\alpha)\in\Phi^+(\fq)\cap A\right\} \\
                                        &\Phi^{-}_\setminus=\left\{\alpha\in\Phi^-(\fq'):w^{-1}(\alpha)\in\Phi^-(\fq)\setminus A\right\}, \quad 
                                        \Phi^{-}_\cap=\left\{\alpha\in\Phi^-(\fq'):w^{-1}(\alpha)\in\Phi^-(\fq)\cap A\right\}.
                                \end{align*}
                                
                                Then
                                \begin{eqnarray*}
                                        2w^{-1}\cdot\mu(w(A)) & = & \sum_{\alpha\in \Phi^+ (\fq')\setminus w(A)}w^{-1}(\alpha)-\sum_{\alpha\in w(A)}w^{-1}(\alpha) \\
                                        & = & \sum_{\alpha \in \Phi^{+}_\setminus}w^{-1}(\alpha)+\sum_{\alpha \in \Phi^{-}_\cap}w^{-1}(\alpha) -\sum_{\alpha \in \Phi^{+}_\cap}w^{-1}(\alpha)-\sum_{\alpha \in \Phi^{-}_\setminus}w^{-1}(\alpha) \\
                                        & = & \sum_{\beta\in\Phi^+ (\fq)\setminus A}\beta-\sum_{\beta\in A}\beta = 2\mu(A)
                                \end{eqnarray*}
                                whence the claimed equality. Finally, the map $A\mapsto w(A)$ establishes a~bijection between the sets
                                \[
                                \left\{A\subset\Phi^+ (\fq):\mu(A)=\mu\right\}\quad\mbox{and}\quad
                                \left\{A'\subset\Phi^+ (\fq'):\mu(A')=w(\mu)\right\},
                                \]
                                whence the equality of multiplicities.
                                
                                (b) is a~consequence of (a) and of the fact that ${}^w \spin=\spin$ whenever $w\in W'$.
                                
                                (c) If $A'=\Phi^+ (\fq)\setminus A$ then
                                \[
                                \mu(A')=\frac{1}{2}\Big(\sum_{\alpha\in A}\alpha-\sum_{\alpha\in A'}\alpha\Big)=-\mu(A).
                                \]
                                The map $A\mapsto \Phi^+ (\fq)\setminus A$ establishes a~bijection
                                \[
                                \{A\subset \Phi^+ (\fq):\mu=\mu(A)\}\stackrel{\sim}{\longrightarrow} \{A'\subset\Phi^+ (\fq):-\mu=\mu(A')\}.
                                \]
                                Whence (c).
                        \end{proof}

                        As shown in the lemma, for studying the weights of the spin modules in the case of Levi subalgebras,
                        one can restrict to standard Levi subalgebras $\mathfrak{h}$ among a~set of representatives of the $W$-conjugacy classes, and for a~fixed $\mathfrak{h}$ one can restrict to considering weights among a~set of representatives of the $W'$-orbits of $\Lambda(\spin)$.
                        
                        \begin{example}\label{E3.4}
                                
                                \begin{enumerate}[(a)]
                                        \item\label{E3.4a}  Every standard Levi subalgebra of $\gl(n)$ or $\mathfrak{sl}(n)$ is conjugated under $W$ to the subspace of blockwise matrices $\fh=\fh(c)$ determined by a~composition $c=(c_1,\ldots,c_k)$ of $n$, such that $c_1\leq\ldots\leq c_k$.
                                        If $n=4$, for instance, we are left with five Levi subalgebras up to conjugation:
                                        \begin{center}
                                                \begin{tabular}
                                                        {|c|c|c|c|c|c|c|c|}
                                                        \hline
                                                        $c$ & $(1^4)$ & $(1^2,2)$ & $(2,2)$ & $(1,3)$ & $(4)$
                                                        \\ \hline
                                                        $\fh(c)$ &
                                                        {\footnotesize $
                                                                \begin{pmatrix}
                                                                        * & 0 & 0 & 0 \\ 0 & * & 0 & 0 \\ 0 & 0 & * & 0 \\ 0 & 0 & 0 & *
                                                                \end{pmatrix}
                                                                $} &
                                                        {\footnotesize $
                                                                \begin{pmatrix}
                                                                        * & 0 & 0 & 0 \\ 0 & * & 0 & 0 \\ 0 & 0 & * & * \\ 0 & 0 & * & *
                                                                \end{pmatrix}
                                                                $} &
                                                        {\footnotesize $
                                                                \begin{pmatrix}
                                                                        * & * & 0 & 0 \\ * & * & 0 & 0 \\ 0 & 0 & * & * \\ 0 & 0 & * & *
                                                                \end{pmatrix}
                                                                $} &
                                                        {\footnotesize $
                                                                \begin{pmatrix}
                                                                        * & 0 & 0 & 0 \\ 0 & * & * & * \\ 0 & * & * & * \\ 0 & * & * & *
                                                                \end{pmatrix}
                                                                $} &
                                                        {\footnotesize $
                                                                \begin{pmatrix}
                                                                        * & * & * & * \\ * & * & * & * \\ * & * & * & * \\ * & * & * & *
                                                                \end{pmatrix}
                                                                $}
                                                        \\
                                                        \hline
                                                \end{tabular}
                                        \end{center}
                                        
                                        \item\label{E3.4b} Let $\fg=\gl(n)$ with $n\geq 3$, $\fh=\mathfrak{h}(1,n-1)\cong \gl(1)\times\gl(n-1)$, thus
                                        \[
                                        \fq=\left(
                                        \begin{array}
                                                {c|ccc}
                                                0 & * & \cdots & * \\ \hline * & 0 & \cdots & 0 \\ \vdots & \vdots & \ddots & \vdots \\ * & 0 & \cdots & 0
                                        \end{array}
                                        \right)\quad\mbox{and}\quad
                                        \Phi^+ (\fq)=\{\epsilon_1-\epsilon_i:i=2,\ldots,n\}.
                                        \]
                                        In this case, according to Example~\hyperref[E1b]{2.1(b)}, we have $W'=W(\fh,\ft)=\mathfrak{S}_{n-1}$.
                                        For $k=0,\ldots,n-1$, we define
                                        \[
                                        A_k=\{\epsilon_1-\epsilon_i:i=2,\ldots,k+1  \}
                                        \]
                                        which is a~subset of $\Phi^+ (\fq)$ with $k$ elements. Conversely, whenever $A=\{\epsilon_{1}-\epsilon_{i_j}\}_{j=1}^k$ is a~subset of $\Phi^+ (\fq)$ with $k$ elements, there is $w\in W'$
                                        such that $A=w(A_k)$, namely
                                        \[
                                        w=
                                        \begin{pmatrix}
                                                1 & 2 & \cdots & k+1 & k+2 & \cdots & n \\ 1 & i_1 & \cdots & i_k & i'_{k+1} & \cdots & i'_{n-1}
                                        \end{pmatrix}
                                        \]
                                        (whenever $\{2,\ldots,n\}=\{i_1,\ldots,i_k\}\sqcup\{i'_{k+1},\ldots,i'_{n-1}\}$).
                                        
                                        The weight associated to $A_k$ is
                                        \[
                                        \mu(A_k)=\frac{1}{2}\Big((n-2k-1)\epsilon_1+\epsilon_2+\ldots+\epsilon_{k+1}-\epsilon_{k+2}-\ldots-\epsilon_n\Big)
                                        \]
                                        and we have
                                        \[
                                        \Lambda(\mathbf{S})=\bigsqcup_{k=0}^{n-1}W'(\mu(A_k)).
                                        \]
                                        
                                        We also note that the map
                                        \[
                                        A=\{\epsilon_{1}-\epsilon_{i_j}\}_{j=1}^k \mapsto \mu(A)=\frac{1}{2}\big((n-2k-1)\epsilon_1+\epsilon_{i_1}+\ldots+\epsilon_{i_k}-\epsilon_{i'_{k+1}}-\ldots-\epsilon_{i'_{n-1}}\big)
                                        \]
                                        is injective. Therefore, we have $\mult_\spin(\mu(A))=1$ for all weight $\mu(A)\in\Lambda(\spin)$.
                                \end{enumerate}
                        \end{example}
                        
                        \subsection{Inductive properties of multiplicities}\label{section-2.2}
                        
                        The standard parabolic subalgebras of $\fg$ can be parametrized by subsets of simple roots.
                        Namely, a~subset $I\subset\Delta$
                        of simple roots gives rise to a~parabolic subalgebra $\mathfrak{p}_I\subset \fg$ and a~standard Levi factor $\mathfrak{l}_I$. We also denote by $\Phi_I:=\Phi\cap\langle I\rangle$ and $\Phi_I^+ :=\Phi^+ \cap\langle I\rangle$ the root system and set of positive roots for $\mathfrak{l}_I$, where $\Phi=\Phi(\fg)$ and $\Phi^+$ are as above.
                        
                        Finally, let $\spin_I=\spin(\fl_I,\fl_I\cap\fh,\fl_I\cap\fq^+)$ be the spin module relative to $\mathfrak{l}_I$ and its Levi subalgebra $\fl_I\cap\fh$. For every subset $A\subset \Phi^+_I(\fq):=\Phi^+_I\cap\Phi^+ (\fq)$, let $\mu_{I}(A)\in\Lambda(\spin_I)$ be the corresponding weight, that is
                        \[
                        \mu_{I}(A)=\frac{1}{2}\Big(\sum_{\alpha\in \Phi^+_I(\fq)\setminus A}\alpha-\sum_{\alpha\in A}\alpha\Big)=\rho(\fl_I\cap\fq)-\sum_{\alpha\in A}\alpha
                        \]
                        where $\rho(\fl_I\cap\fq)=\frac{1}{2}\sum_{\alpha\in\Phi_I^+ (\fq)}\alpha$.
                        Since $A$ is also a~subset of $\Phi^+ (\fq)$, the weight $\mu(A)\in\Lambda(\spin)$ and the multiplicity $\mult_\spin(\mu(A))$ can be considered. The following lemma relates the weights $\mu_{I}(A)$ and $\mu(A)$ and their multiplicities.
                        
                        \begin{lemma}\label{L3.7-new}
                                Let $I\subset\Delta$ be a~subset of simple roots and let $A\subset\Phi_I^+ (\fq)$.
                                Then we have $\mult_\spin(\mu(A))=\mult_{\spin_I}(\mu_{I}(A))$.
                        \end{lemma}
                        
                        \begin{proof}
                                By \eqref{spinweight}, it suffices to show the equality of sets
                                \[
                                \left\{\{A'\subset\Phi_I^+ (\fq):\sum_{\alpha\in A'}\alpha=\sum_{\alpha\in A}\alpha\right\}=\left\{\{A'\subset\Phi^+ (\fq):\sum_{\alpha\in A'}\alpha=\sum_{\alpha\in A}\alpha\right\}.
                                \]
                                Since the inclusion $\subset$ is immediate, we only have to show the reverse inclusion. So let $A'\subset\Phi^+ (\fq)$ be a~subset such that $\sum_{\alpha\in A'}\alpha=\sum_{\alpha\in A}\alpha$. Since $\sum_{\alpha\in A}\alpha\in\langle I\rangle$, every simple root which arises as a~summand of a~root in $A'$ must belong to $I$. This implies that $A'\subset\Phi_I^+ (\fq)$ and therefore, the claim follows.
                        \end{proof}
                        
                        We finally give a~multiplicative formula for the multiplicity in the case where the subset $A$ can be decomposed as the union of subsets lying in disconnected subsystems of roots.
                        
                        \begin{proposition}\label{L3.8-new}
                                Let $I_1,\ldots,I_k\subset\Delta$ be subsets of simple roots such that $I_i\perp I_j$ whenever $i\not=j$. Assume that $A=\bigcup_{j=1}^k A_j$ where $A_j\subset \Phi^+_{I_j}(\fq)$. Then,
                                \[
                                \mult_\spin(\mu(A))=\prod_{j=1}^k \mult_\spin(\mu(A_j)).
                                \]
                        \end{proposition}
                        
                        \begin{proof}
                                It is sufficient to deal with the situation where $k=2$, i.e., $A=A_1\cup A_2$; after that, the result follows from an easy induction on $k\geq 2$.
                                By \eqref{spinweight}, we know that
                                \[
                                \mult_\spin(\mu(A))=\#\left\{A'\subset\Phi^+ (\fq):\sum_{\alpha\in A'}\alpha=\sum_{\alpha\in A}\alpha\right\}.
                                \]
                                We claim that there is a~bijection
                                \begin{eqnarray*}
                                        &\displaystyle \left\{A'_1\subset \Phi_{I_1}^+ (\fq):\sum_{\alpha\in A'_1}\alpha=\sum_{\alpha\in A_1}\alpha\right\}\times \left\{A'_2\subset \Phi_{I_2}^+ (\fq):\sum_{\alpha\in A'_2}\alpha=\sum_{\alpha\in A_2}\alpha\right\} \\
                                        &\displaystyle \longrightarrow\ \left\{A'\subset\Phi^+ (\fq):\sum_{\alpha\in A'}\alpha=\sum_{\alpha\in A}\alpha\right\}
                                \end{eqnarray*}
                                given by $(A'_1,A'_2)\mapsto A'_1\cup A'_2$. This map is clearly well defined. It is injective since $A'_i=(A'_1\cup A'_2)\cap\langle I'_i\rangle$
                                for all $i\in\{1,2\}$. For the surjectivity, since 
                                \[\sum_{\alpha\in A'}\alpha=\sum_{\alpha\in A}\alpha\in\langle I_1\cup I_2\rangle=\langle I_1\rangle\oplus\langle I_2\rangle,\]
                                then $A'\subset \Phi^+_{I_1\cup I_2}(\fq)=\Phi^+_{I_1}(\fq)\cup \Phi_{I_2}^+ (\fq)$, hence $A'=A'_1\cup A'_2$ with $A'_i\subset\Phi^+_{I_i}(\fq)$, and we have 
                                \[\sum_{\alpha\in A'} \alpha=\sum_{\alpha\in A'_1} \alpha+\sum_{\alpha\in A'_2} \alpha=\sum_{\alpha\in A_1} \alpha+\sum_{\alpha\in A_2} \alpha,\]
                                hence $\sum_{\alpha\in A'_i} \alpha=\sum_{\alpha\in A_i}\alpha$ for $i\in\{1,2\}$, by considering the direct sum $\langle I_1\rangle\oplus\langle I_2\rangle$. This establishes the desired bijection. Thereby, $\mult_\spin(\mu(A))=\mult_\spin(\mu(A_1))\cdot\mult_\spin(\mu(A_2))$, by \eqref{spinweight} and Lemma~\ref{L3.7-new}.
                        \end{proof}

                        \subsection{General properties of weights in the case of a~Cartan subalgebra}\label{section-2.3}
                        
                        In this section, we discuss the multiplicities of weights of the spin module $\spin$ when $\fh=\ft$ is a~Cartan subalgebra.
                        Note that there is a~partial order on weights, specifically we set $\mu\leq\mu'$ if $\mu'-\mu\in\sum_{\alpha\in\Phi^+}\mathbb{Q}_{\geq 0}\alpha$.
                        
                        \begin{lemma}\label{L1}
                                \begin{enumerate}[(a)]
                                        \item\label{L1a} $\rho(\fg)=\mu(\emptyset)\in\Lambda(\spin)$.
                                        \item\label{L1b} For all $A\subset\Phi^+$, we have $\mu(A)=\rho(\fg)-\sum_{\alpha\in A}\alpha$.
                                        In particular, $\mu\leq\rho(\fg)$ for all $\mu\in\Lambda(\spin)$.
                                        \item\label{L1c} $\mult_\spin(\rho(\fg))=1$.
                                \end{enumerate}
                        \end{lemma}
                        
                        \begin{proof}
                                Parts \ref{L1a} and \ref{L1b} are immediate in view of the definitions of $\mu(\emptyset)$ and $\mu(A)$.
                                It follows from \ref{L1b} that we have $\mu(A)=\rho(\fg)$ if and only if $A=\emptyset$, so that part \ref{L1c} ensues.
                        \end{proof}
                        
                        As pointed out in Example~\hyperref[E1a]{2.1(a)}, in the present case of $\fh=\ft$, we have $W'=W$ (where $W'$ is as in Section~\ref{S3.1}). By Lemma~\ref{L-w-Levi}, $W$ acts on $\Lambda(\spin)$ in a~natural way and this action preserves the multiplicities of weights.
                        
                        By $\alpha^\vee\in\ft$ we denote the coroot associated to $\alpha\in\Phi^+$.
                        Recall that a~weight $\mu$ is said to be \textit{dominant} if
                        \[
                        \langle\mu,\alpha^\vee\rangle\geq 0\quad \mbox{for all}\ \alpha\in\Delta.
                        \]
                        We have $\langle\rho(\fg),\alpha^\vee\rangle=1$ for all simple root $\alpha$, which implies that $\rho(\fg)$ is a~dominant weight.
                        
                        By the action of $W$, every weight $\omega$ of $\ft$, hence \textit{a~fortiori} every weight $\mu\in\Lambda(\spin)$, can be transformed into a~dominant weight.
                        Thereby, we can restrict our attention to dominant weights of $\spin$.
                        
                        As noted in Example~\hyperref[E3.4b]{3.4(b)}, in the general case of Levi subalgebras, it can happen that all the weights of $\spin$ are of multiplicity one. We now stress that the situation is different in the case of a~Cartan subalgebra.
                        
                        \begin{definition}
                                We say that a~subset $A\subset\Phi^+$ is \emph{saturated} if the following condition is satisfied:
                                \[
                                \forall \alpha,\beta\in A,\ \alpha+\beta\in\Phi\Rightarrow \alpha+\beta\in A;\quad\mbox{and}\quad\forall \gamma,\delta\in \Phi^+ \setminus A,\ \gamma+\delta\in\Phi \Rightarrow \gamma+\delta\notin A.
                                \]
                        \end{definition}
                        
                        \begin{proposition}\label{pro-saturated}
                                Let $\mu\in\Lambda(\spin)$. Let $A\subset\Phi^+$ be a~subset such that $\mu=\mu(A)$. The following conditions are equivalent:
                                \begin{enumerate}[(i)]
                                        \item $\mult_\spin(\mu)=1$;
                                        \item $A$ is saturated;
                                        \item There is $w\in W$ such that $\mu=w(\rho(\fg))$.
                                \end{enumerate}
                        \end{proposition}
                        
                        \begin{proof}
                                By Lemma~\ref{L1}, we have $\mult_\spin(\rho(\fg))=1$. The implication (iii)$\Rightarrow$(i) is obtained by
                                combining this observation with Lemma~\ref{L-w-Levi}.
                                
                                (i)$\Rightarrow$(ii): Arguing indirectly, assume that $A$ is not saturated.
                                If there are $\alpha,\beta\in A$ such that $\gamma:=\alpha+\beta\in \Phi^+ \setminus A$, then $A':=A\cup\{\gamma\}\setminus\{\alpha,\beta\}$ is a~subset different from $A$ such that $\mu(A')=\mu(A)=\mu$. If there are $\gamma,\delta\in \Phi^+ \setminus A$ such that $\alpha:=\gamma+\delta\in A$, then $A'':=A\cup\{\gamma,\delta\}\setminus\{\alpha\}$ is different from $A$ and such that $\mu(A'')=\mu(A)=\mu$. In both cases, we conclude that $\mathrm{mult}_\spin(\mu)\geq 2$.
                                
                                (ii)$\Rightarrow$(iii): We argue by induction on $\#A$.
                                If $\#A=0$, i.e., $A=\emptyset$, then $\mu=\rho(\fg)$ and (iii) holds with $w=1$.
                                Let $\#A=k\geq 1$ and
                                assume that the claim is satisfied for any subset $A'$ with $\#A'<k$. We claim that $A$ contains at least one simple root $\alpha\in\Delta$. Indeed, we can choose an element $\delta\in A$ which is minimal for the partial order $\leq$ on roots determined by the set of positive roots. If $\delta$ is not a~simple root, then there is a~simple root $\alpha$ and a~positive root $\gamma$ such that $\delta=\alpha+\gamma$. Then, by minimality of $\delta$, we have $\alpha,\gamma\notin A$, and this contradicts the assumption that $A$ is saturated.
                                
                                Hence we can choose a~simple root $\alpha$ which belongs to $A$. Invoking Lemma~\ref{L1}, it follows that
                                \[
                                \mu=\rho(\fg)-\sum_{\beta\in A}\beta=(\rho(\fg)-\alpha)-\sum_{\beta\in A\setminus\{\alpha\}}\beta=s_\alpha(\rho(\fg))-\sum_{\beta\in A'}s_\alpha(\beta)=s_\alpha(\mu(A'))
                                \]
                                where we set
                                \[
                                A'=s_\alpha(A\setminus \{\alpha\}).
                                \]
                                Note that the simple reflection $s_\alpha$ permutes the positive roots distinct from $\alpha$, and this implies that $A'$ is also a~subset of $\Phi^+$, so that the notation $\mu(A')$ makes sense. In addition $\#A'=k-1$. We claim that
                                \begin{equation}\label{3.3}
                                        \mbox{$A'$ is saturated.}
                                \end{equation}
                                Once \eqref{3.3} is verified, we get the desired conclusion that $\mu=s_\alpha(\mu(A'))\in W(\rho(\fg))$, due to the induction hypothesis. Therefore, it remains to show (\ref{3.3}).
                                
                                To do this, we first assume $\gamma,\delta\in A'$ such that $\gamma+\delta$ is a~root. Since $\gamma+\delta$ is not a~simple root, we have $\gamma+\delta\not=\alpha$, hence $s_\alpha(\gamma+\delta)$ is a~positive root, and we have
                                \[s_\alpha(\gamma+\delta)=s_\alpha(\gamma)+s_\alpha(\delta),\]
                                where $s_\alpha(\gamma)$ and $s_\alpha(\delta)$ belong to $A$. Since $A$ is saturated, this implies that $s_\alpha(\gamma+\delta)\in A$. Moreover, $s_\alpha(\gamma+\delta)\not=\alpha$ (because $s_\alpha(\alpha)\in-\Phi^+$), whence $\gamma+\delta\in s_\alpha(A\setminus\{\alpha\})=A'$. 
                                
                                We next assume $\gamma,\delta\in\Phi^+ \setminus A'$ such that $\gamma+\delta$ is a root, and let us show that $\gamma+\delta\notin A'$. If $\gamma,\delta$ are different from $\alpha$, then $s_\alpha(\gamma),s_\alpha(\delta)\in\Phi^+ \setminus A$, and this implies that 
                                \[s_\alpha(\gamma+\delta)=s_\alpha(\gamma)+s_\alpha(\delta)\notin A\]
                                due to the fact that $A$ is saturated, and therefore, $\gamma+\delta\notin A'$. 
                                
                                Finally assume that $\gamma\not=\alpha$ and $\delta=\alpha$. We have to show that $\gamma+\alpha$ does not belong to $A'$. Arguing by contradiction, assume that $\gamma+\alpha\in A'$, hence $s_\alpha(\gamma+\alpha)\in A$.
                                Note that 
                                \[s_\alpha(\gamma+\alpha)+\alpha=s_\alpha(\gamma)-\alpha+\alpha=s_\alpha(\gamma)\]
                                is a~root. Since
                                $\alpha,s_\alpha(\gamma+\alpha)\in A$, we must have $s_\alpha(\gamma)\in A$ (because $A$ is saturated), and so $\gamma\in A'$, a~contradiction.
                        \end{proof}
                        
                        \begin{remark}
                                $0$ is a~weight of $\spin$ if and only if $\rho(\fg)$ can be expressed as a~sum of (pairwise distinct) positive roots. This is not always the case: for instance, when $\fg=\mathfrak{sl}(n)$, we have $0\in\Lambda(\spin)$ if and only if $n$ is odd. Note that, given any subset $A\subset\Phi^+$, the following equivalences hold:
                                \[
                                \mu(A)=0\quad\Leftrightarrow\quad\sum_{\alpha\in A}\alpha=\rho(\fg)\quad\Leftrightarrow\quad \sum_{\alpha\in\Phi^+ \setminus A}\alpha=\rho(\fg).
                                \]
                                In particular, the mapping $A\mapsto \Phi^+ \setminus A$ induces an involution without fixed point on the set $\{A\subset\Phi^+ :\mu(A)=0\}$, which implies that $\mult_\spin(0)$ is always even.
                        \end{remark}

                        \section{The case of a~Levi subalgebra for \texorpdfstring{$\mathfrak{g}=\mathfrak{sl}(n)$}{g=sln}}\label{section-4}
                        
                        In this section we assume that $\fg=\mathfrak{sl}(n)$ and $\ft$ is the Cartan subalgebra consisting of diagonal matrices $\text{diag}(t_1,t_2,\ldots,t_n)$ with $t_1+t_2+\ldots+t_n=0$. For $i\in\{1,\ldots,n\}$, let $\epsilon_i\in\ft^*$ be defined by
                        \[
                        \epsilon_i:\text{diag}(t_1,t_2,\ldots,t_n)\mapsto t_i.
                        \]
                        We choose the following sets of positive roots and simple roots :
                        \begin{eqnarray*}
                                \Phi^+ &=&\{\epsilon_i-\epsilon_j\mid1\leq i<j\leq n\}\quad\mbox{with}\quad \#\Phi^+ =\frac{n(n-1)}{2},\\
                                \Delta & = & \{\alpha_1,\ldots,\alpha_{n-1}\}\quad\mbox{where}\quad\alpha_i=\epsilon_i-\epsilon_{i+1}.
                        \end{eqnarray*}
                        %The weight lattice is $\Lambda=\bigoplus_{i=1}^n \mathbb{Z}\epsilon_i$.
                        The $i$th fundamental weight is $\varpi_i=\epsilon_1+\ldots+\epsilon_i$ ($i=1,\ldots,n-1$).
                        The weight lattice is $\Lambda=\bigoplus_{i=1}^{n-1}\mathbb{Z}\varpi_i$.
                        A weight $\mu=\sum_{i=1}^{n-1} x_i\varpi_i$ is dominant if $x_i\geq 0$ for all $i$.
                        The half-sum of positive roots is
                        \[
                        \rho(\fg)=\frac{n-1}{2}\epsilon_1+\frac{n-3}{2}\epsilon_2+\ldots-\frac{(n-3)}{2}\epsilon_{n-1}-\frac{(n-1)}{2}\epsilon_n
                        =\varpi_1+\ldots+\varpi_{n-1}.
                        \]
                        Note that if $\mu=b_1\epsilon_1+\ldots+b_n\epsilon_n$ is a~weight of $\spin$, then it must satisfy $b_1+\ldots+b_n=0$ (see \eqref{spinweight}).
                        
                        The computation of multiplicities of weights for general modules often reduces to challenging combinatorial problems. In the present case of ${\mathfrak g}={\mathfrak s}{\mathfrak l}(n)$, our strategy is based on a~specific parametrization of weights involving fairly simple matrices associated with the root system of ${\mathfrak s}{\mathfrak l}(n)$. It would be interesting to extend this method to other classical Lie algebras and root systems.
                        
                        \subsection{Encoding weights with matrices}\label{S4.1}
                        
                        As seen in \eqref{spinweight}, any weight $\mu\in\Lambda(\mathbf{S})$ is of the form $\mu(A)$ where $A\subset\Phi^+ (\mathfrak{q})$.
                        
                        Let $\fh=\fh(c)$ be the Levi subalgebra corresponding to a~composition
                        $c=(c_1\leq \ldots\leq c_k)$ of $n$ (see Example~\ref{E3.4}\,(a)).
                        We define the sequence
                        \begin{equation}\label{sigma}
                                \sigma=((n-c_1)^{c_1},\ldots,(n-c_k)^{c_k})
                        \end{equation}
                        which is a~composition of $n^2 -(c_1^2 +\ldots+c_k^2)$ (in fact a~partition).

                        We denote by $\mathcal{M}(\fq)$ the set of matrices $a=(a_{i,j})$ of size $n\times n$ which have the following property:
                        \[
                        \left\{
                        \begin{tabular}
                                {ll}
                                $\bullet$ & $a_{i,j}\in\{0,1\}$ for all $i,j$; \\
                                $\bullet$ & $a_{i,j}=1$ $\Rightarrow$ $a_{j,i}=0$; \\
                                $\bullet$ & $\epsilon_i-\epsilon_j\in\Phi(\fh)$ $\Rightarrow$ $a_{i,j}=0$; \\
                                $\bullet$ & $\epsilon_i-\epsilon_j\in\Phi(\fq)$ $\Rightarrow$ $a_{i,j}=1$ or $a_{j,i}=1$.
                        \end{tabular}
                        \right.
                        \]
                        The second point yields in particular $a_{i,i}=0$ for all $i$.
                        
                        For every $A\subset\Phi^+ (\fq)$, we can define a~matrix $a\in\mathcal{M}(\fq)$ by letting
                        \[
                        a_{i,j}=\left\{
                        \begin{array}
                                {ll} 1 & \mbox{if $\epsilon_i-\epsilon_j\in \Phi^+ (\fq)\setminus A$ or $\epsilon_j-\epsilon_i\in A$}, \\ 0 & \mbox{otherwise}.
                        \end{array}
                        \right.
                        \]
                        Conversely, if $a=(a_{i,j})$ belongs to $\mathcal{M}(\fq)$, then we define the set
                        \[
                        A=\{\epsilon_i-\epsilon_j\in\Phi^+ (\fq):a_{i,j}=0\}.
                        \]
                        
                        \begin{lemma}\label{L-matrices}
                                We have 
                                \[\mu(A)=\mu_a:=\sum_{i=1}^n \left(a_{i,*}-\frac{\sigma_i}{2}\right)\epsilon_i=\sum_{i=1}^{n-1}\left(a_{i,*}-a_{i+1,*}-\frac{\sigma_i-\sigma_{i+1}}{2}\right)\varpi_i,\]
                                where $a_{i,*}:=\sum_{j=1}^n a_{i,j}$ and $\sigma_i$ is the $i$th term of the sequence~(\ref{sigma}).
                        \end{lemma}

                        \begin{proof}
                                We have
                                \begin{align*}
                                        2\mu(A) &= \sum_{\alpha\in\Phi^+ (\fq)\setminus A}\alpha -\sum_{\beta\in A}\beta \\
                                        &= \sum_{1\leq i<j\leq n} a_{i,j}(\epsilon_i-\epsilon_j)-\sum_{1\leq j<i\leq n}a_{i,j}(\epsilon_j-\epsilon_i)  \\
                                        &=\sum_{i=1}^n \sum_{j=1}^n a_{i,j}\epsilon_i-\sum_{i=1}^n \sum_{j=1}^n a_{i,j}\epsilon_j  \\
                                        &= \sum_{i=1}^n a_{i,*}\epsilon_i-\sum_{j=1}^n a_{*,j}\epsilon_j  \\
                                        &=\sum_{i=1}^n a_{i,*}\epsilon_i-\sum_{i=1}^n (\sigma_i-a_{i,*})\epsilon_i \\
                                        &= \sum_{i=1}^n (2a_{i,*}-\sigma_i)\epsilon_i 
                                \end{align*}
                                where $a_{*,j}$ stands for the sum of coefficients in the $j$-th column of the matrix, and where we note that $a_{i,*}+a_{*,i}=\sigma_i$, due to the properties of $a$ and the definition of the sequence $\sigma$ in (\ref{sigma}).
                        \end{proof}
                        
                        \begin{example}
                                We consider the Levi subalgebra $\fh=\fh(1,1,2)\subset\mathfrak{sl}(4)$ seen in Example~\ref{E3.4}\,(a). In this case, we have
                                $\sigma=(3,3,2,2)\vdash 10$. We encode a~weight $\mu=\sum_{i=1}^4 b_i\epsilon_i$ with the sequence $(b_1,b_2,b_3,b_4)$. Then
                                \begin{gather*}
                                        \mu_{a_1} =(\frac{3}{2},\frac{1}{2},-1,-1), \quad \mu_{a_2}=(\frac{3}{2},-\frac{1}{2},0,-1), \quad\mu_{a_3}=(\frac{3}{2},-\frac{3}{2},0,0),\\
                                        \mu_{a_4}=(\frac{1}{2},-\frac{1}{2},0,0),\quad\mu_{a_5}=(\frac{1}{2},\frac{1}{2},0,-1),\quad \mu_{a_6}=(\frac{1}{2},-\frac{3}{2},1,0),\\
                                        \mu_{a_7}=(\frac{1}{2},-\frac{1}{2},1,-1),\quad \mu_{a_8}=(-\frac{1}{2},-\frac{1}{2},1,0),\quad  \mu_{a_9}=(-\frac{1}{2},-\frac{3}{2},1,1)
                                \end{gather*}
                                are weights of $\spin$
                                corresponding to the matrices
                                \begin{eqnarray*}
                                        &a_1=\mbox{\footnotesize $
                                                \begin{pmatrix}
                                                        0 & 1 & 1 & 1 \\ 0 & 0 & 1 & 1 \\ 0 & 0 & 0 & 0 \\ 0 & 0 & 0 & 0
                                                \end{pmatrix}
                                                $},
                                        \
                                        a_2=\mbox{\footnotesize $
                                                \begin{pmatrix}
                                                        0 & 1 & 1 & 1 \\ 0 & 0 & 0 & 1 \\ 0 & 1 & 0 & 0 \\ 0 & 0 & 0 & 0
                                                \end{pmatrix}
                                                $},
                                        \
                                        a_3=\mbox{\footnotesize $
                                                \begin{pmatrix}
                                                        0 & 1 & 1 & 1 \\ 0 & 0 & 0 & 0 \\ 0 & 1 & 0 & 0 \\ 0 & 1 & 0 & 0
                                                \end{pmatrix}
                                                $},
                                        \\
                                        & a_4=\mbox{\footnotesize $
                                                \begin{pmatrix}
                                                        0 & 1 & 1 & 0 \\ 0 & 0 & 0 & 1 \\ 0 & 1 & 0 & 0 \\ 1 & 0 & 0 & 0
                                                \end{pmatrix}
                                                $},
                                        \
                                        a_5=\mbox{\footnotesize $
                                                \begin{pmatrix}
                                                        0 & 1 & 0 & 1 \\ 0 & 0 & 1 & 1 \\ 1 & 0 & 0 & 0 \\ 0 & 0 & 0 & 0
                                                \end{pmatrix}
                                                $},
                                        \
                                        a_6=\mbox{\footnotesize $
                                                \begin{pmatrix}
                                                        0 & 1 & 0 & 1 \\ 0 & 0 & 0 & 0 \\ 1 & 1 & 0 & 0 \\ 0 & 1 & 0 & 0
                                                \end{pmatrix}
                                                $},
                                        \\
                                        & a_7=\mbox{\footnotesize $
                                                \begin{pmatrix}
                                                        0 & 1 & 0 & 1 \\ 0 & 0 & 0 & 1 \\ 1 & 1 & 0 & 0 \\ 0 & 0 & 0 & 0
                                                \end{pmatrix}
                                                $},
                                        \
                                        a_8=\mbox{\footnotesize $
                                                \begin{pmatrix}
                                                        0 & 1 & 0 & 0 \\ 0 & 0 & 0 & 1 \\ 1 & 1 & 0 & 0 \\ 1 & 0 & 0 & 0
                                                \end{pmatrix}
                                                $},
                                        \
                                        a_9=\mbox{\footnotesize $
                                                \begin{pmatrix}
                                                        0 & 1 & 0 & 0 \\ 0 & 0 & 0 & 0 \\ 1 & 1 & 0 & 0 \\ 1 & 1 & 0 & 0
                                                \end{pmatrix}
                                                $}.
                                \end{eqnarray*}
                                In the present case, we have
                                $W'=\{w\in\mathfrak{S}_4:w(\{1,2\})=\{1,2\}\}\cong\mathfrak{S}_2\times\mathfrak{S}_2$ (where $W'$ is introduced in Section~\ref{S3.1}), and every weight $\mu\in\Lambda(\spin)$ is in the $W'$-orbit of exactly one of the elements $\mu_{a_1},\ldots,\mu_{a_9}$.
                        \end{example}

                        \subsection{Inductive formula}\label{section-inductive}
                        We aim to give an inductive formula for the multiplicities.
                        To do this, it is convenient to adopt the following notation: if $(c_1,\ldots,c_k)$ is a~composition of $n$ and $(b_1,\ldots,b_n)$ is a~sequence of half integers such that $b_1+\ldots+b_n=0$, then let
                        \[
                        \mult_{(c_1,\ldots,c_k)}(b_1,\ldots,b_n):=\mult_\spin(\mu)
                        \]
                        be the multiplicity of the weight $\mu=\sum_{i=1}^n b_i\epsilon_i$, relatively to the Levi subalgebra
                        $\fh=\fh(c)$ corresponding to the composition $c=(c_1\leq\ldots\leq c_k)$.
                        When $X$ is a~set, let $\mathcal{P}_{k}(X)$ be the set of subsets $J\subset X$ with $k$ elements.
                        Moreover, given $J\subset X$, let $\mathbf{1}_J:X\to \{0,1\}$ be the function given by $\mathbf{1}_J(x)=1$ or $0$ depending on whether $x\in J$ or $x\notin J$.

                        \begin{theorem}\label{T-inductive-Levi}
                                With the above notation, we have
                                \begin{align*}
                                        &\mult_{(c_1,\ldots,c_k)}(b_1,\ldots,b_n)  \\
                                        &=\sum_{J}\mult_{(c_1-1,c_2,\ldots,c_k)}\Big(b_2,\ldots,b_{c_1},b_{c_1+1}+\mathbf{1}_J(c_1+1)-\frac{1}{2},\ldots,b_n+\mathbf{1}_J(n)-\frac{1}{2}\Big)
                                \end{align*}
                                where $J$ runs over $\mathcal{P}_{b_1+\frac{n-c_1}{2}}(\{c_1+1,\ldots,n\})$ in the sum.
                        \end{theorem}

                        \begin{proof}
                                Let $\sigma=(\sigma_i)_{i=1}^n =((n-c_1)^{c_1},\ldots,(n-c_k)^{c_k})$ be as before. Note that
                                \[
                                \sum_{i=1}^n \sigma_i=n^2 -\sum_{j=1}^k c_j^2 =\dim\fq.
                                \]
                                We use the set $\mathcal{M}(\fq)$ of matrices introduced in Section~\ref{S4.1} and recall that
                                \[
                                \mult_\spin(\mu)=N_\mu:=\#\{a\in\mathcal{M}(\fq):\mu_a=\mu\}.
                                \]
                                Any matrix $a\in\mathcal{M}(\fq)$ contains exactly $\frac{n^2 -\sum c_j^2}{2}$ coefficients equal to $1$.
                                More precisely, for all $i\in\{1,\ldots,n\}$, we have
                                \[
                                0\leq a_{i,*}\leq \sigma_i.
                                \]
                                Note that
                                \[
                                \mu_a=\mu\Leftrightarrow \forall i\in\{1,\ldots,n\},\ a_{i,*}-\frac{\sigma_i}{2}=b_i.
                                \]
                                In particular, we must have $a_{1,*}=b_1+\frac{n-c_1}{2}$. On the other hand,   
                                we have $a_{1,*}\in[0,n-c_1]\cap\mathbb{Z}$. Hence
                                \[
                                b_1\notin\{-\frac{n-c_1}{2},-\frac{n-c_1}{2}+1,\ldots,\frac{n-c_1}{2}\}\Rightarrow \mult_\spin(\mu)=0.
                                \]
                                Now, assume that
                                \[
                                k_1:=b_1+\frac{n-c_1}{2}\in\{0,\ldots,n-c_1\}.
                                \]
                                Given $J\in \mathcal{P}_{k_1}(\{c_1+1,\ldots,n\})$, we want to calculate the number
                                $N_\mu(J)=\#\mathcal{M}_\mu(J)$, where
                                \[
                                \mathcal{M}_\mu(J):=\left\{a=\mbox{\small$\left(
                                        \begin{array}
                                                {c|ccccc} 0 & \cdots0 & \!\alpha_{c_1+1}\!\! & \cdots & \!\!\alpha_n \\ \hline \vdots & & & \\ 0 & & & \\ \!\!\!\beta_{c_1+1}\! & & a' & \\ \vdots & & & \\ \beta_{n} & & &
                                        \end{array}
                                        \right)$}\in\mathcal{M}(q):\mu_a=\mu,\ \alpha_j=1-\beta_j=\mathbf{1}_J(j)\right\}.
                                \]
                                Note that
                                \[
                                N_\mu=\sum_{J\in\mathcal{P}_{k_1}(\{c_1+1,\ldots,n\})}N_\mu(J).
                                \]
                                We are looking for the condition on $a'$ which ensures $a\in \mathcal{M}_\mu(J)$.
                                Note that
                                $a_{i,*}=\beta_i+a'_{i,*}$ for all $i$, where we additionally set $\beta_2=\ldots=\beta_{c_1}=0$.
                                Hence
                                \begin{align*}
                                        a\in\mathcal{M}_\mu(J)&\Leftrightarrow \forall i=2,\ldots,n,\ a_{i,*}=b_i+\frac{\sigma_i}{2} \\
                                        &\Leftrightarrow \forall i=2,\ldots,n,\ a'_{i,*}=b_i-\beta_i+\frac{\sigma_i}{2}.    
                                \end{align*}
                                
                                Let $c'=(c_1-1,c_2,\ldots,c_k)$ and consider the corresponding
                                \[
                                \sigma'=(\sigma'_i)_{i=2}^n =((n-c_1)^{c_1-1},(n-1-c_2)^{c_2},\ldots,(n-1-c_k)^{c_k}).
                                \]
                                Hence
                                \[
                                \sigma'_i=\left\{
                                \begin{array}
                                        {ll} \sigma_i & \mbox{if $2\leq i\leq c_1$}, \\ \sigma_i-1 & \mbox{if $c_1+1\leq i\leq n$.}
                                \end{array}
                                \right.
                                \]
                                Thus
                                \begin{align*}
                                        a\in\mathcal{M}_\mu(J)&\Leftrightarrow \forall i=2,\ldots,n,\ a'_{i,*}-\frac{\sigma'_i}{2}=b_i-\beta_i+\frac{\sigma_i-\sigma'_i}{2}\\*
                                        &\Leftrightarrow \mu_{a'}=\sum_{i=2}^n (b_i-\beta_i+\frac{\sigma_i-\sigma'_i}{2})\epsilon_i.
                                \end{align*}
                                Note that 
                                \begin{align*}
                                        b_i-\beta_i&=%\left\{
                                        \begin{cases}
                                                b_i & \mbox{if $i\in J$ or $i\in\{2,\ldots,c_1\}$}, \\ b_i-1 & \mbox{if $i\in \{c_1+1,\ldots,n\}\setminus J$},
                                        \end{cases}
                                        \text{ and}\\
                                        \frac{\sigma_i-\sigma'_i}{2}&=%\left\{
                                        \begin{cases}
                                                0 & \mbox{if $2\leq i\leq c_1$}, \\ \frac{1}{2} & \mbox{if $c_1<i\leq n$}.
                                        \end{cases}
                                \end{align*}
                                %\right\$..
                                This establishes the result.
                        \end{proof}

                        \begin{example}
                                \begin{enumerate}[(a)]
                                        \item If $c=(n)$ then $\mult_\spin(\mu)=0$ if $\mu\not=0$ and $\mult_\spin(0)=1$.
                                        (In this case, the spin module is the trivial module.)
                                        \item If $c=(1,n-1)$ then the inductive formula yields
                                        \[
                                        \mult_{(1,n-1)}(b_1,\ldots,b_n)=\left\{
                                        \begin{array}
                                                {ll}1 & \mbox{if $b_1+b_2+\ldots+b_n=0$ and $b_i\in\{\pm\frac{1}{2}\}$ for all $i\geq 2$}, \\ 0 & \mbox{otherwise.}
                                        \end{array}
                                        \right.
                                        \]
                                        so we retrieve the fact that every nonzero multiplicity is $1$; see Example~\hyperref[E3.4b]{2.3(b)}.
                                \end{enumerate}
                        \end{example}

                        \section{The case of a~Cartan subalgebra for \texorpdfstring{$\mathfrak{g}=\mathfrak{sl}(n)$}{g=sln}}\label{section-Cartan-sln}
                        In this section we assume that ${\mathfrak h}=\ft$ and we keep the notation of Section~\ref{section-4}.
                        
                        \subsection{Encoding dominant weights of \texorpdfstring{$\spin$}{S} with partitions}\label{section-5.1}
                        
                        \begin{definition}\label{D-partitions}
                                \begin{enumerate}[(a)]
                                        \item For $n\geq 1$, we
                                        denote by $\mathrm{Part}(\binom{n}{2};n)$ the set of partitions of $\binom{n}{2}$ with at most $n$ parts, i.e., sequences of nonnegative integers
                                        \[
                                        \lambda=(\lambda_1\geq\ldots\geq\lambda_n)\quad\mbox{with}\quad\quad \lambda_1+\ldots+\lambda_n=\binom{n}{2}.
                                        \]
                                        \item Let $\mathcal{P}(n)\subset\mathrm{Part}(\binom{n}{2};n)$ be the subset of partitions
                                        satisfying in addition the condition
                                        \[
                                        \lambda_1+\ldots+\lambda_i\leq (n-1)+\ldots+(n-i)\ \mbox{ for all $i\in\{1,\ldots,n\}$}.
                                        \]
                                \end{enumerate}
                        \end{definition}
                        
                        Note that if $n=1$, we get $\mathrm{Part}(\binom{n}{2};n)=\mathrm{Part}(0;1)=\mathcal{P}(1)$.
                        
                        Recall the dominance order on partitions: given $\lambda=(\lambda_1,\ldots,\lambda_n)$,
                        $\lambda'=(\lambda'_1,\ldots,\lambda'_n)$ in $\mathrm{Part}(\binom{n}{2};n)$, we set $\lambda\preceq\lambda'$ if
                        \[
                        \lambda_1+\ldots+\lambda_i\leq\lambda'_1+\ldots+\lambda'_i\ \mbox{for all $i\geq1$}.
                        \]
                        Thus, $\mathcal{P}(n)$ is the set of partitions of $\binom{n}{2}$ with at most $n$ parts, which are $\preceq\lambda^0$ for the partition
                        \[
                        \lambda^0 :=(n-1,n-2,\ldots,1,0).
                        \]

                        \begin{example}
                                Let $n=5$. The elements of $\mathcal{P}(5)$ correspond to the Young diagrams
                                \begin{eqnarray*}
                                        &\mbox{\footnotesize$\yng(4,3,2,1)$}=\lambda^0,\quad
                                        \mbox{\footnotesize$\yng(4,3,1,1,1)$},\quad\mbox{\footnotesize$\yng(4,2,2,2)$},\quad
                                        \mbox{\footnotesize$\yng(4,2,2,1,1)$},\\
                                        &\mbox{\footnotesize$\yng(3,3,3,1)$},\quad
                                        \mbox{\footnotesize$\yng(3,3,2,2)$},\quad
                                        \mbox{\footnotesize$\yng(3,3,2,1,1)$},\quad
                                        \mbox{\footnotesize$\yng(3,2,2,2,1)$},\quad
                                        \mbox{\footnotesize$\yng(2,2,2,2,2)$}.
                                \end{eqnarray*}
                        \end{example}
                        
                        We define the dominant weight associated to a~partition $\lambda\in\mathrm{Part}(\binom{n}{2};n)$:
                        
                        \begin{definition}\label{mulambda}
                                Given $\lambda=(\lambda_1,\ldots,\lambda_n)\in\mathrm{Part}(\binom{n}{2};n)$, we define
                                \[
                                \mu[\lambda]=(\lambda_1-\lambda_2)\varpi_1+\ldots+(\lambda_{n-1}-\lambda_n)\varpi_{n-1}.
                                \]
                                In particular, $\mu[\lambda^0 ]=\rho(\fg)$.
                        \end{definition}
                        
                        \begin{remark}\label{R-matrix-partition}
                                In the present case where $\fh=\ft$ is a~Cartan subalgebra, the set $\mathcal{M}(\fq)$
                                of Section~\ref{S4.1} consists of matrices $a\in\mathcal{M}_n(\{0,1\})$ with $a_{i,i}=0$ and $a_{i,j}+a_{j,i}=1$ for all $i,j\in\{1,\ldots,n\}$, $i\not=j$.
                                Such a~matrix $a$ corresponds to a~subset 
                                \[A=\{\epsilon_i-\epsilon_j:i<j,\ a_{i,j}=0\}\subset\Phi^+\]
                                whose associated weight is
                                $\mu(A)=\sum_{i=1}^n (a_{i,*}-\frac{n.1}{2})\epsilon_i$,
                                where $a_{i,*}$ stands for the sum of coefficients in the $i$-th row of $a$
                                (see Lemma~\ref{L-matrices}).
                                This weight is dominant if and only if $a_{1,*}\geq\ldots\geq a_{n,*}$. This is equivalent to saying that the sequence $\lambda:=(a_{1,*},\ldots,a_{n,*})$ is a~partition of $\binom{n}{2}$. Moreover, in this case, we have $\mu(A)=\mu[\lambda]$.
                        \end{remark}
                        
                        Recall that the Weyl group $W=\mathfrak{S}_n$ acts on $\Lambda(\spin)$ (and this action preserves the multiplicity of weights: see Example~\ref{E1}\,(a) and Lemma~\ref{L-w-Levi}) and the $W$-orbit of every weight $\mu\in\Lambda(\spin)$ contains exactly one dominant weight. Thus, there is a~bijection $\Lambda(\spin)/W\cong\Lambda^+ (\spin)$, where $\Lambda^+ (\spin)\subset\Lambda(\spin)$ stands for the subset of dominant weights.
                        
                        \begin{proposition}\label{P4.4}
                                The map $\lambda\mapsto\mu[\lambda]$ establishes a~bijection between the set of partitions $\mathcal{P}(n)$ and the subset $\Lambda^+ (\spin)$ of dominant weights of the spin module.
                        \end{proposition}
                        
                        The proposition is shown in Section~\ref{section-4.3} below. In Sections~\ref{section-4.3} and~\ref{section-4.4}, we calculate the multiplicities of weights, $\mult_\spin(\mu[\lambda])$, in terms of the corresponding partitions $\lambda$.
                        Before this, in the following subsection, we introduce some combinatorial material.
                        
                        \subsection{Inductive formula for multiplicities in terms of partitions}\label{section-4.3}

                        For a~partition $\lambda=(\lambda_1,\ldots,\lambda_n)\in\mathrm{Part}(\binom{n}{2};n)$, let
                        \[
                        N_\lambda=\mult_\spin(\mu[\lambda]).
                        \]
                        We deduce from Theorem~\ref{T-inductive-Levi} an inductive formula for computing the numbers $N_\lambda$.
                        To this end, we introduce the following notation.
                        
                        \begin{notation}\label{notation-markings}
                                Let $\lambda=(\lambda_1,\ldots,\lambda_n)\in\mathrm{Part}(\binom{n}{2};n)$ be a~partition viewed as a~Young diagram.
                                Fix an index $p\in\{1,\ldots,n\}$.
                                We call \emph{$p$-marking} the datum of a~subset of boxes $\beta\subset\lambda$ such that
                                \begin{itemize}
                                        \item $\beta$ contains exactly $n-1$ boxes;
                                        \item $\beta$ contains all the boxes of the $p$-th row of $\lambda$;
                                        \item in the other rows of $\lambda$, only the rightmost box may belong to $\beta$.
                                \end{itemize}
                                Let $M_p(\lambda)$ be the set of $p$-markings of $\lambda$.
                                For a $p$-marking $\beta\in M_p(\lambda)$,
                                by arranging the lengths of the rows of the subset $\lambda\setminus \beta\subset\lambda$ in nonincreasing order,
                                we obtain a~partition of $\binom{n}{2}-(n-1)=\binom{n-1}{2}$, which we denote by $\lambda-\beta$.
                        \end{notation}
                        
                        \begin{theorem}\label{P2}
                                Given a~partition $\lambda=(\lambda_1,\ldots,\lambda_n)\in\mathrm{Part}(\binom{n}{2};n)$ with at most $n$ positive parts
                                and any $p\in\{1,\ldots,n\}$,
                                we have
                                \[
                                \mult_\spin(\mu[\lambda])=N_\lambda=\sum_{\beta\in M_p(\lambda)} N_{\lambda-\beta}.
                                \]
                        \end{theorem}
                        
                        \begin{proof}
                                \begin{eqnarray*}
                                        N_\lambda & = & \mult_\spin(\mu[\lambda]) \\
                                        & = & \mult_\spin\big((\lambda_1-\lambda_2)\varpi_1+\ldots+(\lambda_{n-1}-\lambda_n)\varpi_{n-1}\big) \\
                                        & = & \mult_{(1,1,\ldots,1)}\Big(\lambda_1-\frac{n-1}{2},\ldots,\lambda_n-\frac{n-1}{2}\Big) \\
                                        & = & \mult_{(1,1,\ldots,1)}\Big(\tilde\lambda_1-\frac{n-1}{2},\ldots,\tilde\lambda_n-\frac{n-1}{2}\Big),\\ && \qquad\qquad\qquad\qquad \mbox{where $(\tilde\lambda_1,\ldots,\tilde\lambda_n)=(\lambda_p,\lambda_1,\ldots,\lambda_{p-1},\lambda_{p+1},\ldots,\lambda_n)$} \\
                                        % & = & \mult_{(1,1,\ldots,1)}\Big(\lambda_p-\frac{n-1}{2},\lambda_1-\frac{n-1}{2},\ldots,\lambda_{p-1}-\frac{n-1}{2},\lambda_{p+1}-\frac{n-1}{2},\ldots,\lambda_n-\frac{n-1}{2}\Big) %\\
                                        & = & \sum_{J\in\mathcal{P}_{\tilde\lambda_1}(\{2,\ldots,n\})}
                                        \mult_{(1,\ldots,1)}\Big(\tilde\lambda_2-\frac{n}{2}+\mathbf{1}_J(2),\ldots,\tilde\lambda_n-\frac{n}{2}+\mathbf{1}_J(n)\Big).
                                \end{eqnarray*}
                                There is a~bijection $M_p(\lambda)\cong \mathcal{P}_{\tilde\lambda_1}(\{2,\ldots,n\})$.
                                To see this, we first set $\tilde{\lambda}=(\tilde\lambda_1,\ldots,\tilde\lambda_n)$, which is not necessarily a~partition (since it is not necessarily a~nonincreasing sequence) but can nevertheless be represented by a~set of boxes with $n$ rows of lengths $\tilde\lambda_1,\ldots,\tilde\lambda_n$.
                                
                                Now, if $\beta$ is a~$p$-marking of $\lambda$, then $\beta$ is a~subset of $\lambda$ that contains exactly $n-1$ elements, so that $\beta$ contains all the $\lambda_p=\tilde\lambda_1$ boxes of the $p$-th row of $\lambda$ (which corresponds to the first row of $\tilde\lambda$) and exactly one box in $n-1-\tilde\lambda_1$ rows among the rows number $1,\ldots,p-1,p+1,\ldots,n$ of $\lambda$, that is among the rows number $2,\ldots,n$ of $\tilde\lambda$. Now we set 
                                \[J:=\{j=2,\ldots,n:\mbox{$\beta$ does not contain any box of the $j$-th row of $\tilde\lambda$}\},\]
                                which is therefore a~subset of $\{2,\ldots,n\}$ with $\tilde\lambda_1$ elements, that is $J\in\mathcal{P}_{\tilde\lambda_1}(\{2,\ldots,n\})$. Conversely, given $J\in\mathcal{P}_{\tilde\lambda_1}(\{2,\ldots,n\})$, let $\beta$ be the $p$-marking of $\lambda$ that contains the boxes of the $p$-th row of $\lambda$ (i.e., the first row of $\tilde\lambda$) and the rightmost box of the $i$-th row of $\tilde\lambda$ whenever $i\in\{2,\ldots,n\}\setminus J$.
                                
                                Assume that $J\in\mathcal{P}_{\tilde\lambda_1}(\{2,\ldots,n\})$ corresponds to the $p$-marking $\beta$ as above. Then, for  $i=2,\ldots,n$.
                                \[
                                \tilde\lambda_i-\frac{n}{2}+\mathbf{1}_J(i)=
                                \begin{cases}
                                        \tilde\lambda_i-1-\frac{n-2}{2} & \text{ if } i \notin J, \\
                                        \tilde\lambda_i-\frac{n-2}{2} & \text{ otherwise.}
                                \end{cases}
                                \]
                                We note that the condition $i \notin J$ means that $\beta$ contains the rightmost box of the $i$-th row of $\tilde\lambda$. The above formula implies that
                                \[
                                \Big(\tilde\lambda_2-\frac{n}{2}+\mathbf{1}_J(2),\ldots,\tilde\lambda_n-\frac{n}{2}+\mathbf{1}_J(n)\Big)
                                \]
                                is the list obtained by removing $\frac{n-2}{2}$ from the lengths of the rows of $\lambda-\beta$ (up to reordering).

                                This implies that the weight corresponding to this list is in the same $\mathfrak{S}_{n-1}$-orbit as $\mu[\lambda-\beta]$, so that
                                \[
                                \mult_{(1,\ldots,1)}\Big(\tilde\lambda_2-\frac{n}{2}+\mathbf{1}_J(2),\ldots,\tilde\lambda_n-\frac{n}{2}+\mathbf{1}_J(n)\Big)=N_{\lambda-\beta}.
                                \]
                                The formula now follows.
                        \end{proof}
                        
                        We prove an additional lemma which will be useful.
                        
                        \begin{lemma}\label{lambda'}
                                Let $\lambda=(\lambda_1\geq\ldots\geq\lambda_n)\in\mathcal{P}(n)$. There is a~partition $\lambda'$ of the form $\lambda'=(n-1,\lambda'_2,\ldots,\lambda'_n)$ with $\lambda'_j\in \{\lambda_j-1,\lambda_j\}$ for all $j\in\{2,\ldots,n\}$, such that $\lambda'\preceq (n-1,n-2,\ldots,1)$.
                        \end{lemma}
                        
                        %First, suppose that $\lambda=(\lambda_1\geq\ldots\geq\lambda_n)\in\mathcal{P}(n)$. This means that, 

                        \begin{proof}
                                Firstly, observe that for all $j$, we have 
                                \[\lambda_1+\ldots+\lambda_j\leq (n-1)+(n-2)+\ldots+(n-j).\]
                                Note that $\lambda_{n-1}\not=0$ because otherwise the previous inequality fails for $j=n-2$.
                                To prove the lmma, we show more generally that if $(\mu_1,\ldots,\mu_k)$ is a~nonincreasing sequence of positive integers such that
                                \begin{itemize}
                                        \item $k>(n-1)-\mu_1$,
                                        \item $\mu_1+\ldots+\mu_j\leq (n-1)+\ldots+(n-j)$ for all $j=1,\ldots,k$,
                                \end{itemize}
                                then there is a~nonincreasing sequence $(\mu'_1,\ldots,\mu'_k)$ of positive integers with
                                \[
                                \left\{
                                \begin{array}{ll}
                                        \mbox{\rm (a)} & \mu'_1=n-1, \\
                                        \mbox{\rm (b)} & \mbox{$\mu'_j\in\{\mu_j-1,\mu_j\}$ for all j=2,\ldots,k,} \\
                                        \mbox{\rm (c)} & \mu'_1+\ldots+\mu'_k=\mu_1+\ldots+\mu_k, \\
                                        \mbox{\rm (d)} & \mbox{$\mu'_1+\ldots+\mu'_j\leq (n-1)+\ldots+(n-j)$ for all j=1,\ldots,k.}
                                \end{array}
                                \right.
                                \leqno{(\star)_\mu}
                                \]
                                Applying this to $\mu=\lambda$ and setting $\lambda'=\mu'$, we get the conclusion of the claim. Hence it suffices to justify the construction of $\mu'$. We do this by induction on $(n-1)-\mu_1$. If $(n-1)-\mu_1=0$ then we set $\mu'=\mu$ and we are done. Now assume that $(n-1)-\mu_1>0$. Take $t\in\{1,\ldots,k\}$ maximal such that 
                                \[\mu_1+\ldots+\mu_j<(n-1)+\ldots+(n-j) \text{ for all } j=1,\ldots,t.\]
                                
                                First case: $t=k$. Then we set \[\tilde\mu=(\tilde\mu_1,\ldots,\tilde\mu_{k-1})=(\mu_1+1,\mu_2,\ldots,\mu_{k-1}).\]
                                Since $(n-1)-\tilde\mu_1<(n-1)-\mu_1$, we can invoke the induction hypothesis which provides us with $\tilde\mu'=(\tilde\mu'_1=n-1,\tilde\mu'_2,\ldots,\tilde\mu'_{k-1})$
                                satisfying $(\star)_{\tilde\mu}$.
                                Finally, we set 
                                \[\mu'=(n-1,\tilde\mu'_2,\ldots,\tilde\mu'_{k-1},\mu_{k}-1)\]
                                which fulfils the conditions.
                                
                                Second case: $t<k$. The maximality of $t$ implies that 
                                \[\mu_1+\ldots+\mu_{t+1}=(n-1)+\ldots+(n-t-1).\]
                                Since the left-hand side is $\leq (t+1)\mu_1$, this easily implies that $t>n-1-\mu_1$.
                                Note also that, in the case where $t+1<k$, the above equality combined with $\sum_{j=1}^{t+2}\mu_j\leq\sum_{j=1}^{t+2}(n-j)$ also forces $\mu_{t+2}<\mu_{t+1}$.
                                Now we set
                                \[
                                \tilde\mu=(\tilde\mu_1,\ldots,\tilde\mu_t)=(\mu_1+1,\mu_2,\ldots,\mu_t).
                                \]
                                Due to the choice of $t$, it is immediate that 
                                \[\tilde\mu_1+\ldots+\tilde\mu_j\leq (n-1)+\ldots+(n-j)\text{ for all }j=1,\ldots,t.\]
                                In addition, we have $t>(n-1)-\tilde\mu_1$ as noted above. On this basis, we can invoke the induction hypothesis which yields a~nonincreasing sequence $\tilde\mu'=(\tilde\mu'_1=n-1,\tilde\mu'_2,\ldots,\tilde\mu'_t)$ satisfying $(\star)_{\tilde\mu}$. Finally we set $\mu'=(n-1,\tilde\mu'_2,\ldots,\tilde\mu'_t,\mu_{t+1}-1,\mu_{t+2},\ldots,\mu_k)$ which satisfies the required conditions.
                                
                                This concludes the proof of the lemma. 
                        \end{proof}

                        With the help of Theorem~\ref{P2} and Lemma \ref{lambda'},
                        we are now in position to prove Proposition~\ref{P4.4}.

                        \begin{proof}[Proof of Proposition~\ref{P4.4}]
                                We have to show that, given $\lambda\in\mathrm{Part}(\binom{n}{2};n)$, the following equivalence holds:
                                \[
                                N_\lambda\not=0\quad\Leftrightarrow\quad \lambda\in\mathcal{P}(n).
                                \]
                                We proceed by induction on $n\geq 1$. The base case is clear, so we now assume the result is true up to $n-1$.

                                We consider the partition $\lambda'$ provided by Lemma \ref{lambda'}. Note that 
                                \[(\lambda'_2,\ldots,\lambda'_n)\preceq(n-2,\ldots,1).\]
                                
                                Moreover, if we denote by $\beta\subset\lambda$ (viewed as a~diagram) the subset formed by the first row and the rightmost box of the $j$-th row of $\lambda$ for every $j\in\{2,\ldots,n\}$ such that $\lambda'_j=\lambda_j-1$, then $\beta$ is a~$1$-marking of $\lambda$ and we have $\lambda-\beta=(\lambda'_2,\ldots,\lambda'_n)$. In view of the induction formula (Theorem~\ref{P2}), $N_{(\lambda'_2,\ldots,\lambda'_n)}$ is one of the terms arising in the sum $N_\lambda$. By induction hypothesis
                                $N_{(\lambda'_2,\ldots,\lambda'_n)}>0$, hence $N_\lambda>0$.
                                
                                Next assume that $\lambda\in\mathrm{Part}(\binom{n}{2};n)\setminus\mathcal{P}(n)$.
                                Let $j$ be minimal such that 
                                \[\lambda_1+\ldots+\lambda_j>(n-1)+\ldots+(n-j).\]
                                
                                If $\lambda'=(\lambda'_2,\ldots,\lambda'_n)$ is a~partition of $\binom{n-1}{2}$ obtained from $\lambda$ by considering any $1$-marking, then we must have 
                                \[\lambda'_2+\ldots+\lambda'_j\geq\lambda_1+\ldots+\lambda_j-(n-1)>(n-2)+\ldots+(n-j),\]
                                and therefore $\lambda'\not\preceq(n-2,\ldots,1)$. This yields $N_{\lambda'}=0$ (by induction hypothesis). Since $N_\lambda$ is a~sum of such terms (by Theorem~\ref{P2}), we conclude that $N_\lambda=0$.
                        \end{proof}
                        
                        \begin{example}\label{E4.14}
                                \begin{enumerate}[(a)]
                                        \item Let $\lambda=(3,3,2,1,1)\in\mathcal{P}(5)$, viewed as the Young diagram
                                        \[
                                        \lambda=\yng(3,3,2,1,1).
                                        \]
                                        Let $p=5$. Then, the $5$-markings of $\lambda$ are
                                        \[
                                        \young(\ \ \bullet,\ \ \bullet,\ \bullet,\ ,\bullet),\quad \young(\ \ \bullet,\ \ \bullet,\ \ ,\bullet,\bullet),\quad
                                        \young(\ \ \bullet,\ \ \ ,\ \bullet,\bullet,\bullet),\quad \young(\ \ \ ,\ \ \bullet,\ \bullet,\bullet,\bullet).
                                        \]
                                        Whence the formula
                                        \[
                                        N_\lambda=N_{(2,2,1,1)}+N_{(2,2,2)}+2N_{(3,2,1)}.
                                        \]
                                        
                                        \item\label{E4.14b} We have computed the numbers $N_\lambda$ associated to the partitions $\lambda\in\mathcal{P}(n)$, for $n\leq 5$.
                                        The values are listed below.
                                        \[
                                        N_{(1)}=1.\leqno{(n=2)}
                                        \]
                                        \[
                                        N_{(2,1)}=N_{(1)}=1;\qquad N_{(1^3)}=2N_{(1)}=2.\leqno{(n=3)}
                                        \]
                                        \[
                                        \begin{array}{ll}
                                                N_{(3,2,1)}=N_{(2,1)}=1; & N_{(3,1^3)}=2N_{(2,1)}=2;\\
                                                N_{(2^3)}=N_{(1^3)}=2; & N_{(2^2,1^2)}=N_{(1^3)}+2N_{(2,1)}=4. 
                                        \end{array}   
                                        \leqno{(n=4)}
                                        \]
                                        
                                        \[
                                        \begin{array}{lllll}
                                                N_{(4,3,2,1)}=N_{(3,2,1)}=1; & N_{(4,3,1^3)}=2N_{(3,2,1)}=2; \\  
                                                N_{(4,2^3)}=N_{(3,1^3)}=2; & N_{(4,2^2,1^2)}=2N_{(3,2,1)}+N_{(3,1^3)}=4; \\
                                                N_{(3^3,1)}=N_{(2^3)}=2; &  N_{(3^2,2^2)}=N_{(2^2,1^2)}=4;\\
                                                N_{(3^2,2,1^2)}=N_{(2^2,1^2)}+N_{(3,1^3)}+2N_{(3,2,1)}=8;& N_{(3,2^3,1)}=N_{(2^3)}+3N_{(2^2,1^2)}=14; \\
                                                N_{(2^5)}=\binom{4}{2}N_{(2^2,1^2)}=24. &
                                        \end{array}
                                        \leqno{(n=5)}
                                        \]

                                        % \[
                                        % \begin{cases}
                                                % N_{(4,3,2,1)}=N_{(3,2,1)}=1;\quad N_{(4,3,1^3)}=2N_{(3,2,1)}=2;\quad N_{(4,2^3)}=N_{(3,1^3)}=2;\\
                                                % N_{(4,2^2,1^2)}=2N_{(3,2,1)}+N_{(3,1^3)}=4;\quad N_{(3^3,1)}=N_{(2^3)}=2;\quad N_{(3^2,2^2)}=N_{(2^2,1^2)}=4;\\
                                                % N_{(3^2,2,1^2)}=N_{(2^2,1^2)}+N_{(3,1^3)}+2N_{(3,2,1)}=8;\quad N_{(3,2^3,1)}=N_{(2^3)}+3N_{(2^2,1^2)}=14;\\
                                                % N_{(2^5)}=\binom{4}{2}N_{(2^2,1^2)}=24.
                                                % \end{cases}
                                        % \leqno{(n=5)}
                                        % \]
                                \end{enumerate}
                        \end{example}

                        \begin{remark}\label{L5}
                                In the above example, we observe that the multiplicity is either $1$ or even. In fact, for every $\lambda\in\mathcal{P}(n)$, we always have the following characterization:
                                \begin{itemize}
                                        \item If $\lambda=(n-1,n-2,\ldots,1)$, that is $\lambda=\lambda^0$, then $N_\lambda=1$ (which actually follows from Lemma~\ref{L1}).
                                        \item If $\lambda\not=\lambda^0$, then $N_\lambda$ is even.
                                \end{itemize}
                                Indeed,
                                arguing by induction on the basis of Theorem~\ref{P2}, it suffices to show that when $\lambda\not=\lambda^0$, the number of $1$-markings $\beta\in M_1(\lambda)$ such that $\lambda-\beta=(n-2,n-3,\ldots,1)$ is even. 
                                
                                In fact, if any, there are exactly two such markings $\beta^1$ and $\beta^2$: the existence of a~marking imposes that $\lambda=(\lambda_1,\ldots,\lambda_n)$ with $\lambda_1=n-2$, $\lambda_{j_0}=n-j_0+1$ for a single $j_0\in\nolinebreak\{2,\ldots,n\}$, and $\lambda_j=n-j$ for $j\in\{2,\ldots,n\}\setminus\{j_0\}$. Then $\beta^1$ consists of the first row of $\lambda$ and the rightmost box of the $j_0$-th row, while $\beta^2$ consists of the first row of $\lambda$ and the rightmost box of the $(j_0-1)$-th row.
                        \end{remark}
                        
                        \subsection{Enumerative formula for multiplicities}\label{section-4.4}
                        
                        We give an enumerative formula for the multiplicities $\mult_\spin(\mu)$, or equivalently for the numbers $N_\lambda$.\ We rely on the following combinatorial definition.
                        
                        \begin{definition}\label{D-T}
                                \begin{enumerate}
                                        \item\label{D-Ta} Let $\lambda=(\lambda_1,\ldots,\lambda_n)\in\mathcal{P}(n)$. A \emph{spin tableau}  of shape $\lambda$ is a~tableau $\tau$ of shape $\lambda$ which fulfils the following conditions:
                                        \begin{itemize}
                                                \item For every $i\in\{1,\ldots,n-1\}$, $\tau$ contains $i$ boxes of entry $i$, all located within the first $i+1$ rows;
                                                \item The rows (resp. columns) of $\tau$ are nondecreasing from left to right (resp. top to bottom);
                                                moreover, on the $i$-th row, the entries $\geq i$ are increasing.
                                        \end{itemize}
                                        Let $\mathcal{ST}(\lambda)$ be the set of spin tableaux of shape $\lambda$.
                                        \item\label{D-Tb} We label certain boxes of $\tau$ with a~binomial coefficient:
                                        in the $j$-th column,
                                        for each number $i$ that occurs in the column, we label the last box of entry $i$ with the binomial coefficient $\binom{a}{b}$ where
                                        \begin{itemize}
                                                \item $b=b_{i,j}:=\#\{\mbox{boxes of entry $=i$ within the first $i$ boxes of the column}\}$,
                                                \item $a=a_{i,j}:=b+\#\left\{\begin{array}{c}
                                                        \text{boxes of entry } <i \text{ of the column, whose} \\
                                                        \quad \text{ right neighbor (if any) is }>i
                                                \end{array}\right\}$
                                        \end{itemize}
                                        Finally we define $N_\tau$ as the product of these binomial coefficients
                                        taken over the boxes of the whole tableau $\tau$.
                                \end{enumerate}

                        \end{definition}
                        
                        The definition makes sense whenever $n\geq 1$. Note that the only element of $\mathcal{P}(1)$ is the trivial partition $\lambda=(0)$; the unique element of $\mathcal{ST}(0)$ is then the empty tableau $\tau=\emptyset$, with $N_\emptyset=1$.
                        
                        \begin{theorem}\label{theorem-tableaux}
                                Let $\lambda\in\mathcal{P}(n)$.\ With the above notation, we have
                                \[
                                \mult_\spin(\mu[\lambda])=N_\lambda=\sum_{\tau\in\mathcal{ST}(\lambda)}N_\tau.
                                \]
                        \end{theorem}
                        
                        \begin{proof}
                                We prove the result by induction on $n\geq 1$. For $n=1$, the claim is true in view of the discussion above. Now assume that the formula is established up to $n-1\geq 1$.
                                
                                Recall that $M_n(\lambda)$ denotes the set of $n$-markings of $\lambda$ (see Notation~\ref{notation-markings}).
                                If $\tau\in\mathcal{ST}(\lambda)$, then note that the subset $\beta(\tau)\subset\lambda$ formed by the boxes with entry $n-1$ in $\tau$ is an $n$-marking of $\lambda$. Indeed, the first condition in Definition~\ref{D-T} implies that $\tau$ contains $n-1$ boxes of entry $n-1$ and that every box of the $n$-th row has entry $n-1$, whereas the second condition implies that in each of the first $n-1$ rows only the rightmost box can have entry $n-1$.
                                
                                An $n$-marking $\beta$ of $\lambda$ is such that $\lambda-\beta=\lambda-\beta(\tau)$ if and only if $\beta'$ and $\beta'(\tau)$ have the same number of boxes in each column, where $\beta'\subset\beta$, resp. $\beta'(\tau)\subset\beta(\tau)$, denotes the subset formed by the boxes of the marking which are in the first $n-1$ rows.
                                
                                The number of boxes of $\beta'(\tau)$ in the $j$-th column is the number $b_{n-1,j}$ and the number of boxes of the $j$-th column which are the rightmost box of a~row (apart from the $n$-th row) is $a_{n-1,j}$. This implies that the number of $n$-markings of $\lambda$ such that $\lambda-\beta=\lambda-\beta(\tau)$ is
                                \[
                                \#\{\beta\in M_n(\lambda):\lambda-\beta=\lambda-\beta(\tau)\}=\prod_{j\geq 1}\binom{a_{n-1,j}}{b_{n-1,j}}.
                                \]
                                
                                Note that two tableaux $\tau,\tau'$ yield the same marking $\beta(\tau)=\beta(\tau')$ if and only if they differ only by their subtableaux $\tau|_{< n-1},\tau'|_{< n-1}$ formed by the entries $< n-1$, which in fact belong to $\mathcal{ST}(\lambda-\beta(\tau))$.
                                
                                Moreover, we have
                                \[
                                N_{\tau}=N_{\tau|_{<n-1}}\cdot \prod_{j\geq 1}\binom{a_{n-1,j}}{b_{n-1,j}}.
                                \]
                                
                                Let $\mathcal{ST}(\lambda)=\overline{\tau_1}\sqcup\ldots\sqcup\overline{\tau_k}$ be the decomposition into classes for the equivalence relation defined by letting $\tau\sim\tau'$ if $\beta(\tau)=\beta(\tau')$. For all $i\in\{1,\ldots,k\}$, let $\beta_i=\beta(\tau_i)$. Then
                                
                                \begin{eqnarray*}
                                        \sum_{\tau\in\mathcal{ST}(\lambda)}N_\tau & = & \sum_{i=1}^k \sum_{\tau\in\overline{\tau_i}} N_{\tau} \\*
                                        & = & \sum_{i=1}^k \sum_{\tau\in\overline{\tau_i}} N_{\tau|_{< n-1}}\cdot\#\{\beta\in M_n(\lambda):\lambda-\beta=\lambda-\beta_i\} \\
                                        & = & \sum_{i=1}^k \#\{\beta\in M_n(\lambda):\lambda-\beta=\lambda-\beta_i\}\sum_{\sigma\in \mathcal{ST}(\lambda-\beta_i)}N_\sigma \\
                                        & = & \sum_{\beta\in M_n(\lambda)}\sum_{\sigma\in\mathcal{ST}(\lambda-\beta)}N_\sigma \\
                                        & = & \sum_{\beta\in M_n(\lambda)}N_{\lambda-\beta}\qquad\mbox{(by induction hypothesis)} \\
                                        & = & N_{\lambda}\qquad\mbox{(by Theorem~\ref{P2}).}
                                \end{eqnarray*}
                                This concludes the proof of the theorem.
                        \end{proof}
                        
                        \begin{example}
                                \begin{enumerate}
                                        \item For $n=5$, let $\lambda=(2^5)$. There are two spin tableaux of shape $\lambda$, namely
                                        \[
                                        \tau_1=\young(12,23,34,34,44),\quad \tau_2=\young(13,23,24,34,44).
                                        \]
                                        We have $N_{\tau_1}=\binom{a_{3,2}}{b_{3,2}}\binom{a_{4,2}}{b_{4,2}}=\binom{2}{1}\binom{4}{2}$
                                        and $N_{\tau_2}=\binom{a_{2,1}}{b_{2,1}}\binom{a_{4,2}}{b_{4,2}}=\binom{2}{1}\binom{4}{2}$
                                        (keeping track only of binomial coefficients $>1$), hence $N_\lambda=N_{\tau_1}+N_{\tau_2}=\binom{2}{1}\binom{4}{2}+\binom{2}{1}\binom{4}{2}=24$.
                                        \item For $n=7$, let $\lambda=(3^7)$. The spin tableaux of this shape are
                                        \[
                                        \mbox{\footnotesize \young(123,234,345,456,456,556,666),\quad \young(124,234,345,356,456,556,666),\quad \young(123,245,345,346,456,556,666),\quad \young(124,235,345,346,456,556,666),\quad \young(125,235,345,346,446,556,666),\quad \young(134,234,235,456,456,556,666),\quad
                                                \young(134,234,245,356,456,556,666),\quad
                                                \young(134,235,245,346,456,556,666),\quad \young(135,235,245,346,446,556,666).}
                                        \]
                                        Correspondingly, we get
                                        \begin{eqnarray*}
                                                N_\lambda & = & \Big[\binom{2}{1}\binom{3}{1}+\binom{2}{1}\binom{3}{1}+\binom{2}{1}\binom{2}{1}\binom{3}{2}+\binom{2}{1}\binom{3}{2}\binom{2}{1}\binom{3}{2}+\binom{2}{1}\binom{4}{2} \\
                                                & & \quad +\binom{2}{1}\binom{3}{1}+\binom{2}{1}\binom{3}{1}+\binom{2}{1}\binom{3}{2}\binom{2}{1}\binom{3}{2}+\binom{2}{1}\binom{4}{2}\Big]\binom{6}{3} \\
                                                & = & [6+6+12+36+12+6+6+36+12]\times 20=2640
                                        \end{eqnarray*}
                                        which is therefore the multiplicity of the zero weight in $\spin$.
                                \end{enumerate}
                        \end{example}
                        
                        \subsection{Two special cases}
                        
                        \subsubsection{A multiplicative formula for weights} Our aim is to exploit Proposition~\ref{L3.8-new} in order to obtain the following combinatorial relation.
                        
                        \begin{proposition}\label{P-5.13}
                                Let $\mu\in\mathrm{Part}(\binom{m}{2};m)$ and $\pi\in\mathrm{Part}(\binom{p}{2};p)$ with $\pi_1\leq p+\mu_m$, then
                                \[
                                \textstyle \lambda:=(p+\mu_1,\ldots,p+\mu_m,\pi_1,\ldots,\pi_p)\in\mathrm{Part}(\binom{n}{2};n)
                                \ \ \mbox{where}\ \ n=m+p.
                                \]
                                Moreover, $\lambda\in\mathcal{P}(n)$ if and only if $\mu\in\mathcal{P}(m)$ and $\pi\in\mathcal{P}(p)$,
                                and we have
                                \[
                                N_\lambda=N_\mu\times N_\pi.
                                \]
                        \end{proposition}
                        
                        \begin{proof}
                                Clearly, $\lambda\in\mathrm{Part}(\binom{n}{2};n)$.
                                By Definition~\ref{D-partitions},
                                noting that 
                                \[\sum_{i=1}^m \lambda_i=mp+\binom{m}{2}=\sum_{i=1}^m (n-i),\]
                                the condition that $\lambda\in\mathcal{P}(n)$ is equivalent to having
                                $\sum_{i=1}^j \mu_i\leq\sum_{i=1}^j (m-i)$ for all $j\in\nolinebreak\{1,\ldots,m-1\}$
                                and
                                $\sum_{i=1}^{j-m}\pi_{i}\leq\sum_{i=1}^{j-m}(p-i)$ for all $j\in\{m+1,\ldots,m+p\}$,
                                which is equivalent to saying that $\mu\in\mathcal{P}(m)$ and $\pi\in\mathcal{P}(p)$.
                                
                                The last formula of the proposition is clear if $\lambda\notin\mathcal{P}(n)$ in view of Proposition~\ref{P4.4}. So assume that $\lambda\in\mathcal{P}(n)$, that is $\mu\in\mathcal{P}(m)$ and $\pi\in\mathcal{P}(p)$. 
                                Again, by Proposition~\ref{P4.4}, since $N_\mu\not=0$ and $N_\pi\not=0$, we can find matrices $a^{(1)}\in\mathcal{M}_m(\{0,1\})$ and $a^{(2)}\in\mathcal{M}_p(\{0,1\})$ as in Remark~\ref{R-matrix-partition}, and such that $\mu_i=a^{(1)}_{i,*}$ for all $i=1,\ldots,m$ and $\pi_j=a^{(2)}_{j,*}$ for all $j=1,\ldots,p$. Let
                                \[
                                a=
                                \begin{pmatrix}
                                        a^{(1)} & \mathbf{1} \\ \mathbf{0} & a^{(2)}
                                \end{pmatrix}
                                \]
                                where $\mathbf{1}$ and $\mathbf{0}$ represent blocks whose coefficients are all equal to $1$, resp. $0$. Then $a_{i,*}=\lambda_i$ for all $i=1,\ldots,n$. Moreover, the subset $A\subset \Phi^+$ corresponding to $a$ in the sense of Remark~\ref{R-matrix-partition} is of the form $A=A_1\cup A_2$ where $A_1\subset\{\epsilon_i-\epsilon_j:1\leq i<j\leq m\}$ and $A_2\subset\{\epsilon_i-\epsilon_j:m< i<j\leq n\}$ correspond to $a^{(1)}$ and $a^{(2)}$, respectively. By Proposition~\ref{L3.8-new}, we deduce:
                                \[
                                N_\lambda=\mult_\spin(\mu(A))=\mult_\spin(\mu(A_1))\times\mult_\spin(\mu(A_2))=N_\mu\times N_\pi.
                                \]
                                This concludes the proof.
                        \end{proof}
                        
                        \begin{example}
                                If $\mu=(2,2,1,1)\in\mathcal{P}(4)$ and $\pi=(1,1,1)\in\mathcal{P}(3)$, then we get that $\lambda=(5,5,4,4,1,1,1)\in\mathcal{P}(7)$ and $N_\lambda=N_\mu\times N_\pi=4\times 2=8$.
                        \end{example}
                        
                        \subsubsection{The case of a~shift of \texorpdfstring{$\rho(\fg)$}{p(fg)} by a~single root}
                        
                        Here, we consider a~weight of the form $\mu=\rho(\fg)-\alpha$ where $\alpha\in\Phi^+$ is a~positive root. If $\alpha$ is a~simple root, then $\mu=s_\alpha(\rho(\fg))$, $\mult_\spin(\mu)=\mult_\spin(\rho(\fg))=1$, and $\mu$ is not dominant.
                        In the case where $\alpha$ is not simple, we show that $\mu$ is always dominant and we compute its multiplicity.
                        
                        \begin{proposition}\label{P-5.15}
                                Let $\alpha=\epsilon_i-\epsilon_j$ with $1\leq i<j\leq n$ and $j-i>1$, so that $\alpha$ is not a~simple root.
                                Then $\mu:=\rho(\fg)-\alpha$ is a~dominant weight. Under the notation of Definition~\ref{mulambda},
                                we have $\mu=\mu[\lambda]$ where $\lambda=(\lambda_1,\ldots,\lambda_n)\in\mathrm{Part}(\binom{n}{2};n)$
                                is the partition/Young diagram
                                obtained from $\lambda^0 =(n-1,n-2,\ldots,0)$ by moving one box from the $i$-th row to the $j$-th row. Moreover, $\mult_\spin(\mu)=2^{j-i-1}$.
                        \end{proposition}
                        
                        \begin{proof}
                                In the sense of Remark~\ref{R-matrix-partition}, the weight $\rho(\fg)=\mu[\lambda^0 ]$ corresponds to the matrix $a^0 \in\mathcal{M}(\fq)$ given by $a^0_{k,\ell}=1$ if $k<\ell$ and $a_{k,\ell}=0$ otherwise, while $\mu=\rho(\fg)-(\epsilon_i-\epsilon_j)$ corresponds to the matrix $a\in \mathcal{M}(\fq)$ with $a_{i,j}=0$, $a_{j,i}=1$, and $a_{k,\ell}=a^0_{k,\ell}$ for any other pair $(k,\ell)$. Since $\mu=\mu[\lambda]$ with $\lambda=(a_{1,*},\ldots,a_{n,*})$, we deduce that $\lambda$ is obtained from $\lambda^0$ as described in the statement.

                                It remains to show that $N_\lambda=2^{j-i-1}$, we proceed by induction on $n$. If $i>1$ or $j<n$, then we can apply Theorem~\ref{P2} with $p=1$ or $p=n$, so that the claim follows from the induction hypothesis. Hence it remains to deal with the case where $i=1$, $j=n$, that is, $\lambda=(n-2,n-2,n-3,n-4,\ldots,3,2,1,1)$.
                                We aim to apply Theorem~\ref{P2} with $p=1$. A $1$-marking $\beta$ of $\lambda$ is obtained by marking the first row of $\lambda$ and an additional box in the $q$-th row for $q\in\{2,\ldots,n\}$. Then let $\lambda^{(q)}$ be the Young diagram $\lambda-\beta$ so obtained. We have
                                \[
                                N_\lambda=\sum_{q=2}^n N_{\lambda^{(q)}}.
                                \]
                                For $q\in\{n-1,n\}$, we have $\lambda^{(q)}=(n-2,n-3,\ldots,2,1,0)$ hence $N_{\lambda^{(q)}}=1$ in this case. For
                                $2\leq q\leq n-2$, we have 
                                \[\lambda^{(q)}=(n-2,n-3,\ldots,n-(q-1),n-(q+1),n-(q+1),n-(q+2),\ldots,3,2,1,1),\]
                                hence $\lambda^{(q)}$ is obtained from $(n-2,n-3,\ldots,2,1,0)$ by moving one box from the $(q-1)$-th row to the $(n-1)$-th row. Hence $N_{\lambda^{(q)}}=2^{n-q-1}$ due to the induction hypothesis. This yields
                                \[
                                N_\lambda=2^{n-3}+2^{n-4}+\ldots+2^2 +2+1+1=2^{n-2}
                                \]
                                as claimed. The proof is complete.
                        \end{proof}
                        
                        \begin{remark}
                                All multiplicities ($\not=1$) computed in Example~\hyperref[E4.14b]{4.7(b)} that are powers of 2 are multiplicities of weights of the form $\rho(\fg)-\alpha$ with $\alpha$ a~non-simple positive root.
                        \end{remark}
                        
                        \section{The case of maximal parabolic subalgebras for \texorpdfstring{$\mathfrak{g}=\mathfrak{sl}(n)$}{g=sln}}\label{section6}
                        
                        In this section we assume that $\mathfrak{h}$ is a~Levi subalgebra of a~maximal parabolic subalgebra.
                        Specifically we can assume that
                        \[
                        \mathfrak{h}=\fh(p,q)=\left\{
                        \begin{pmatrix}
                                x & 0 \\ 0 & y
                        \end{pmatrix}
                        :x\in\gl(p),\ y\in\gl(q)\right\}
                        \]
                        and
                        \[
                        \mathfrak{q}=\left\{
                        \begin{pmatrix}
                                0 & z \\ t & 0
                        \end{pmatrix}
                        :z\in\mathcal{M}_{p,q}(\mathbb{C}),\ t\in\mathcal{M}_{q,p}(\mathbb{C})\right\}.
                        \]
                        We use the notation introduced in Section~\ref{S4.1}. Thus, the composition of $n$ associated to $\mathfrak{h}$ is $c=(p,q)$, and we have
                        \[
                        \sigma=(\underbrace{q,\ldots,q}_{\mbox{\scriptsize $p$ times}},\underbrace{p,\ldots,p}_{\mbox{\scriptsize $q$ times}}).
                        \]
                        The set $\mathcal{M}(\mathfrak{q})$ encoding the weight vectors of $\mathbf{S}$ is of the form
                        \[
                        \mathcal{M}(\mathfrak{q})=\left\{a=
                        \begin{pmatrix}
                                0 & (a_{i,j}) \\ (a'_{i,j}) & 0
                        \end{pmatrix}
                        \right\}
                        \]
                        with $(a_{i,j})\in\mathcal{M}_{p,q}(\{0,1\})$ and $(a'_{i,j})=\mathbf{1}_{q,p}-{}^t (a_{i,j})$ where $\mathbf{1}_{q,p}$ stands for the $(q,p)$ sized matrix with all coefficients equal to $1$.
                        Thus, there is a~bijection
                        \[
                        \phi:\mathcal{M}_{p,q}(\{0,1\})\to\mathcal{M}(\mathfrak{q}),\ (a_{i,j})\mapsto
                        \begin{pmatrix}
                                0 & (a_{i,j}) \\ \mathbf{1}_{q,p}-{}^t (a_{i,j}) & 0
                        \end{pmatrix}
                        .
                        \]
                        In this way, the weights and weight vectors of $\spin$ can be associated to the matrices of the set $\mathcal{M}_{p,q}(\{0,1\})$. Note that, contrarily to the case of $\mathcal{M}(\mathfrak{q})$, there is no constraint on the coefficients of elements in $\mathcal{M}_{p,q}(\{0,1\})$, in particular all entries can be equal to $1$.
                        
                        The weight corresponding to $a=\phi((a_{i,j}))$ is
                        \begin{eqnarray*}
                                %\label{4.1}
                                \mu_{a} & = & (a_{1,*}-\frac{q}{2},\ldots,a_{p,*}-\frac{q}{2},a'_{1,*}-\frac{p}{2},\ldots,a'_{q,*}-\frac{p}{2}) \\
                                & = & (a_{1,*}-\frac{q}{2},\ldots,a_{p,*}-\frac{q}{2},p-a_{*,1}-\frac{p}{2},\ldots,p-a_{*,q}-\frac{p}{2}) \\
                                & = & (a_{1,*}-\frac{q}{2},\ldots,a_{p,*}-\frac{q}{2},\frac{p}{2}-a_{*,1},\ldots,\frac{p}{2}-a_{*,q}).
                        \end{eqnarray*}
                        Note that the lists $\alpha=(a_{1,*},\ldots,a_{p,*})$ and $\beta=(a_{*,q},\ldots,a_{*,1})$ are sequences of integers with at most $p$ (resp. $q$) parts.
                        We can also note that, by considering the action of the group $W(\fh,\ft)=\mathfrak{S}_p\times\mathfrak{S}_q$ on the weights, we can restrict our attention to weights of the form
                        \[
                        \mu=(\mu_1,\ldots,\mu_p,\nu_1,\ldots,\nu_q)\quad\mbox{with}\quad \mu_1\geq\ldots\geq\mu_p \ \mbox{and}\ \nu_1\geq\ldots\geq\nu_q
                        \]
                        i.e., weights for which $\alpha$ and $\beta$ above are partitions of the same number, with at most $p$ (resp. $q$) parts
                        (as before we identify a~weight $\mu=b_1\epsilon_1+\ldots+b_n\epsilon_n$ with the $n$-tuple $(b_1,\ldots,b_n)$).
                        
                        This suggests to consider the set
                        \[
                        \mathcal{P}=\bigsqcup_{m=0}^{pq}\{(\alpha,\beta):\mbox{$\alpha$ (resp. $\beta$) partition of $m$ with at most $p$ (resp. $q$) parts}\}
                        \]
                        and the map
                        \[
                        \mathcal{P}\ni(\alpha,\beta)\mapsto\mu(\alpha,\beta):=(\alpha_1-\frac{q}{2},\ldots,\alpha_p-\frac{q}{2},\frac{p}{2}-\beta_q,\ldots,\frac{p}{2}-\beta_1).
                        \]
                        Given $(\alpha,\beta)\in\mathcal{P}$, let
                        \[
                        \mathcal{M}_{p,q}(\alpha,\beta)=\left\{(a_{i,j})\in\mathcal{M}_{p,q}(\{0,1\}):
                        \Big\{\begin{array}{ll}
                                a_{i,*}=\alpha_i & \mbox{for all $i=1,\ldots,p$}\\ a_{*,j}=\beta_j & \mbox{for all $j=1,\ldots,q$}
                        \end{array}\right\}.
                        \]
                        Then we have the following lemma.
                        
                        \begin{lemma}\label{L5.1-new}
                                \begin{enumerate}[(a)]
                                        \item $\mu(\alpha,\beta)\in\Lambda(\mathbf{S})\Leftrightarrow\mathcal{M}_{p,q}(\alpha,\beta)\not=\emptyset$.
                                        \item If the conditions of (a) are satisfied, then $\mult_\spin(\mu(\alpha,\beta))=\#\mathcal{M}_{p,q}(\alpha,\beta)$.
                                \end{enumerate}
                        \end{lemma}
                        
                        Therefore, determining the multiplicities of the weights of the spin module $\spin$ attached to the Levi subalgebra $\fh=\fh(p,q)$ with two blocks,
                        becomes equivalent to the following combinatorial problem:
                        
                        \begin{center}
                                \textit{Determine the number of matrices with coefficients in $\{0,1\}$, for a~fixed size and prescribed sums along rows and columns.}
                        \end{center}

                        In particular, this problem will be addressed in Theorem~\ref{T2-new} below.

                        Recall the dominance order $\preceq$ defined in Section~\ref{section-5.1}. By ${}^t \alpha$ we denote the dual partition of $\alpha=(\alpha_1,\ldots,\alpha_p)$, i.e., ${}^t \alpha=(\alpha'_1,\ldots,\alpha'_{\alpha_1})$ where
                        \[
                        \alpha'_j=\#\{i=1,\ldots,p:\alpha_i\geq j\}.
                        \]
                        
                        \begin{proposition}\label{P2-new}
                                Let $(\alpha,\beta)\in\mathcal{P}$.
                                \begin{enumerate}[(a)]
                                        \item If $\beta={}^t \alpha$, then $\mult_\spin(\mu(\alpha,\beta))=1$.
                                        \item If $\beta\prec{}^t \alpha$, then $\mult_\spin(\mu(\alpha,\beta))>1$.
                                        \item If $\beta\not\preceq{}^t \alpha$, then $\mult_\spin(\mu(\alpha,\beta))=0$.
                                \end{enumerate}
                        \end{proposition}
                        
                        \begin{proof}
                                This follows from Proposition~\ref{L5-new} and Theorem~\ref{T2-new} below.
                        \end{proof}
                        
                        \begin{corollary}\label{C5.3}
                                Let $\mathcal{P}'=\{(\alpha,\beta)\in\mathcal{P}:\beta\preceq{}^t \alpha\}$.
                                Let $\Lambda'(\mathbf{S})\subset\Lambda(\mathbf{S})$ be the subset of weights $\mu=(\mu_1,\ldots,\mu_p,\nu_1,\ldots,\nu_q)$ satisfying the condition
                                \[
                                \mu_1\geq\ldots\geq\mu_p,\quad\nu_1\geq\ldots\geq\nu_q
                                \]
                                Then the map
                                \[
                                \mathcal{P}'\to\Lambda'(\mathbf{S}),\ (\alpha,\beta)\mapsto\mu(\alpha,\beta)
                                \]
                                is bijective.
                        \end{corollary}
                        
                        \begin{remark}\label{lambdaprime}
                                Every weight in $\Lambda(\spin)$ is obtained from a~weight in $\Lambda'(\spin)$ by the action of an element in $W(\fh,\ft)=\mathfrak{S}_p\times\mathfrak{S}_q$. Thus, the corollary determines also the elements of $\Lambda(\spin)$.
                        \end{remark}
                        
                        \begin{example}
                                Let $p=q=2$. There are 9 partitions of numbers $\leq 4$ with at most $2$ parts:
                                
                                \[ (4), \, (3,1), \,(2,2), \,(3), \,(2,1), \,(2), \,(1,1), \,(1), \, \emptyset. \]
                                
                                We have
                                \begin{gather*}
                                        {}^t (4)=(1^4), \, {}^t (3,1)=(2,1,1), \,{}^t (2,2)=(2,2), \,{}^t (3)=(1,1,1),  \\ 
                                        {}^t (2,1)=(2,1), \,{}^t (2)=(1,1), \,{}^t (1,1)=(2), \,{}^t (1)=(1), \,{}^t \emptyset=\emptyset.    
                                \end{gather*}
                                
                                We have in this case
                                \begin{align*}
                                        \mathcal{P}'= &\{((2,2),(2,2)),\ ((2,1),(2,1)),\ ((2),(1,1)),\\
                                        &\quad((1,1),(1,1)),\ ((1,1),(2)),\ ((1),(1)),\ (\emptyset,\emptyset)\}.    
                                \end{align*}
                                The set of weights $\mu(\alpha,\beta)$ associated to the elements $(\alpha,\beta)\in\mathcal{P}'$ is correspondingly
                                \begin{gather*}\Lambda'(\spin)=\{(1,1,-1,-1),\ (1,0,0,-1),\ (1,-1,0,0),\ (0,0,0,0)\\
                                        \qquad \qquad \qquad (0,0,1,-1),\ (0,-1,1,0),\ (-1,-1,1,1)\}\end{gather*}
                                
                                The elements of $\Lambda(\spin)$ are obtained from those in $\Lambda'(\spin)$ by permuting the sequences through the action of $W(\fh,\ft)=\mathfrak{S}_2\times\mathfrak{S}_2$, thus
                                
                                \begin{gather*}
                                        \Lambda(\spin)=\Lambda'(\spin)\cup\{(0,1,0,-1),\, (0,1,-1,0),\, (1,0,-1,0),\,(-1,1,0,0),\, \\
                                        \qquad  \qquad \qquad \qquad \qquad  (0,0,-1,1),\, (-1,0,1,0),\, (-1,0,0,1),\, (0,-1,0,1)\}.
                                \end{gather*}
                                
                                Finally, for each pair $(\alpha,\beta)\neq((1,1),(1,1))$ in the set $\mathcal{P}'$, we have $\beta={}^t \alpha$ hence every nonzero weight of $\spin$ has multiplicity one. Moreover, $\mathrm{mult}_\spin(0)=\#\mathcal{M}_{2,2}((1,1),(1,1))=2$.
                                
                                %(b) Let $p=3$, $q=2$. There are 7 partitions of $pq=6$ with at most 3 parts: $(6)$, $(5,1)$, $(4,2)$, $(3,3)$, $(4,1,1)$, $(3,2,1)$, $(2,2,2)$, and there are 4 partitions of $6$ with at most $2$ parts (the first four partitions of the previous list).
                                %We have
                                %${}^t (6)=(1^6)$, ${}^t (5,1)=(2,1^4)$, ${}^t (4,2)=(2,2,1,1)$, ${}^t (3,3)=(2,2,2)$, ${}^t (4,1,1)=(3,1^3)$, ${}^t (3,2,1)=(3,2,1)$, ${}^t (2,2,2)=(3,3)$.
                                %In this case the set $\mathcal{P}'$ reduces to the singleton $\{(\alpha,\beta)=((2,2,2),(3,3))\}$
                                %and the weight associated to the sole element of $\mathcal{P}'$ is $\mu(\alpha,\beta)=(1,1,1,-\frac{3}{2},-\frac{3}{2})$
                                %with multiplicity one (since $\beta={}^t \alpha$). Since this weight is invariant by the action of $W_L=\mathfrak{S}_3\times\mathfrak{S}_2$, we get that $\mu(\alpha,\beta)$ is the only weight of $\spin$ in this case.
                                %Thus $\Lambda(\spin)=\Lambda'(\spin)=\{(1,1,1,-\frac{3}{2},-\frac{3}{2})\}$.
                        \end{example}
                        
                        To show Proposition~\ref{P2-new}, we rely on the following combinatorial notion.
                        
                        \begin{definition}\label{RTtableau}
                                If $(\alpha,\beta)$ is a~pair of partitions of the same size, we call row-tableau of shape $\alpha$ and weight $\beta$
                                a numbering of the boxes of $\alpha$ (viewed as a~Young diagram) such that the numbering comprises of $\beta_i$ occurrences of $i$ for all integer $i$, and such that the entries increase from left to right along the rows. Let $RT(\alpha,\beta)$ denote the set of row-tableaux of shape $\alpha$ and weight $\beta$.
                        \end{definition}
                        
                        \begin{example}
                                \begin{enumerate}[(a)]
                                        \item If $\alpha=(2,2,2)$ and $\beta=(3,3)$ then the only row-tableau of shape $\alpha$ and weight $\beta$ is
                                        \[
                                        \young(12,12,12)
                                        \]
                                        \item  The tableau
                                        \[
                                        \young(124,356,123,14,2)
                                        \]
                                        is a~row-tableau of shape $\alpha=(3,3,3,2,1)$ and weight $\beta=(3,3,2,2,1,1)$.
                                \end{enumerate}
                        \end{example}
                        
                        \begin{proposition}\label{L5-new}
                                Let $(\alpha,\beta)\in\mathcal{P}$. Then:
                                \begin{enumerate}[(a)]
                                        \item $RT(\alpha,\beta)\not=\emptyset\Rightarrow\beta\preceq{}^t \alpha$;
                                        \item $\beta={}^t \alpha\Rightarrow\#RT(\alpha,\beta)=1$;
                                        \item $\beta\prec{}^t \alpha\Rightarrow\#RT(\alpha,\beta)\geq 2$.
                                \end{enumerate}
                                %
                                % \begin{itemize}
                                        % %\item[\rm (a)] $RT(\alpha,\beta)$ is nonempty if and only if $\beta\preceq{}^t \alpha$.
                                        % %\item[\rm (b)] $\#RT(\alpha,\beta)=1$ if and only if $\beta={}^t \alpha$.
                                        % %
                                        % \end{itemize}
                        \end{proposition}
                        
                        \begin{proof}
                                (a) If $T$ is an element of $RT(\alpha,\beta)$, for each number $j\in\{1,\ldots,q\}$ each occurrence of it is located within the first $j$ columns of $T$. This implies that
                                \[
                                \#\{\mbox{entries $\leq j$ in $T$}\}\leq\#\{\mbox{boxes in the first $j$ columns of $T$}\}
                                \]
                                which exactly means that
                                \[
                                \beta_1+\ldots+\beta_j\leq \alpha'_1+\ldots+\alpha'_j\quad\mbox{for all $j=1,\ldots,q$}
                                \]
                                (where ${}^t \alpha=(\alpha'_1,\ldots,\alpha'_{q})$),
                                whence $\beta\preceq{}^t \alpha$.
                                
                                (b) If $\beta={}^t \alpha$, then the number of $j$'s in $T$ coincides with the length of the $j$-th column, hence $T$ must be the unique tableau such that the entries of the $j$-th column are all equal to $j$, for all $j$.
                                
                                (c) We first note that if $\beta\prec{}^t \alpha$, then $\beta$ is obtained from ${}^t \alpha$ by a~sequence of operations of the form $(\beta_1,\ldots,\beta_q)\mapsto (\beta'_1,\ldots,\beta'_{q'})$ (with $q'=q$ or $q+1$)
                                where there are $1\leq j<k\leq q'$ such that
                                \[ \beta'_j=\beta_j-1, \quad \beta'_k=\beta_k+1, \quad \text{and } \beta'_\ell=\beta_\ell \text{ if } \ell\notin\{j,k\}. \] 
                                Then, for showing (c), it is sufficient to show that if $\beta'$ is obtained from $\beta$ by such an operation, then every element $T\in RT(\alpha,\beta)$ induces two elements of $RT(\alpha,\beta')$. Now, since $\beta'$ is still nonincreasing, we must have $\beta_j-\beta_k\geq 2$.
                                Hence there are at least two rows of $T$ which contain an entry $j$ and no entry $k$. Choosing any of these two rows, by replacing the entry $j$ of that row by $k$ and rearranging the row in increasing order, we get a~new tableau which belongs to $RT(\alpha,\beta')$.
                                The proof of the proposition is complete.
                        \end{proof}
                        
                        \begin{example}
                                To illustrate the last step of the proof above, take for instance the partitions $\alpha=(3,3,2,1,1)$, $\beta=(3,2,2,2,1)$, $\beta'=(2,2,2,2,2)$:
                                \[
                                T=\young(124,135,23,4,1)\in RT(\alpha,\beta)\quad\mbox{yields}\quad
                                \young(245,135,23,4,1),\quad \young(124,135,23,4,5)\in RT(\alpha,\beta').
                                \]
                        \end{example}
                        
                        \begin{theorem}\label{T2-new}
                                Let $(\alpha,\beta)\in\mathcal{P}$. Then, $\mult_\spin(\mu(\alpha,\beta))=\#\mathcal{M}_{p,q}(\alpha,\beta)=\#RT(\alpha,\beta)$.
                        \end{theorem}
                        
                        \begin{proof}
                                The first equality is shown in Lemma~\ref{L5.1-new}.
                                To prove the second equality, we define a~map $\mathcal{M}_{p,q}(\alpha,\beta)\to RT(\alpha,\beta)$ in the following way. 
                                
                                Given a~matrix $a=(a_{i,j})\in\mathcal{M}_{p,q}(\alpha,\beta)$, we let $T(a)$ be the tableau obtained in the following way:
                                
                                For $i\in\{1,\ldots,p\}$, there are $\alpha_i$ coefficients of the $i$-th row of $a$ which are equal to $1$. Let $c_{i,1}<\ldots<c_{i,\alpha_i}$ be the column numbers of these coefficients. Then, we put the numbers $c_{i,1},\ldots,c_{i,\alpha_i}$ in the $i$-th row of $T(a)$. We illustrate this procedure in \ref{lastex}.
                                
                                % {\footnotesize
                                        
                                        %    For instance, for $\alpha=(3,3,2,1,1)$ and $\beta=(2,2,2,2,1,1)$:
                                        % \[
                                        % a=
                                        % \begin{pmatrix}
                                                % 0 & 1 & 1 & 1 & 0 & 0\ \\ 1 & 1 & 0 & 0 & 0 & 1 \\ 0 & 0 & 1 & 0 & 1 & 0 \\ 0 & 0 & 0 & 1 & 0 & 0 \\ 1 & 0 & 0 & 0 & 0 & 0
                                                % \end{pmatrix}
                                        % \in\mathcal{M}_{5,6}(\alpha,\beta)
                                        % \quad\Rightarrow\quad T(a)=\young(234,126,35,4,1)\in \mathrm{RT}(\alpha,\beta).
                                        % \]
                                        
                                        % }
                                
                                In this way $T(a)$ is certainly a~tableau of shape $\alpha$, where the entries appear in increasing order in each row. Moreover, for every $j\in\{1,\ldots,q\}$, there are $\beta_j$ instances of the number $1$ in the $j$-th column of $a$, which implies that there are exactly $\beta_j$ row numbers $i$ for which $j\in\{c_{i,1},\ldots,c_{i,\alpha_i}\}$. Hence $T(a)$ contains exactly $\beta_j$ entries equal to $j$. This shows that $T(a)\in RT(\alpha,\beta)$. Hence the map $a\mapsto T(a)$ is well defined.
                                
                                Conversely, we define a~map $RT(\alpha,\beta)\to\mathcal{M}_{p,q}(\alpha,\beta)$. Given $T\in RT(\alpha,\beta)$, let $a_T$ be the matrix of shape $(p,q)$ defined in the following way:
                                For $(i,j)\in\{1,\ldots,p\}\times\{1,\ldots,q\}$, let
                                \[
                                (a_T)_{i,j}=\left\{
                                \begin{array}
                                        {ll}1 & \mbox{if the entry $j$ appears in the $i$-th row of $T$,} \\ 0 & \mbox{otherwise.}
                                \end{array}
                                \right.
                                \]
                                In this way, the number of $1$'s in the $i$-th row of $a_T$ coincides with the number of entries in the $i$-th row of $T$, which is equal to $\alpha_i$,
                                and the number of $1$'s in the $j$-th column of $a_T$ coincides with the number of $j$'s in $T$, which is equal to $\beta_j$. Hence, we have $a_T\in\mathcal{M}_{p,q}(\alpha,\beta)$.
                                
                                It is now straightforward to check that the maps $a\mapsto T(a)$ and $T\mapsto a_T$ are mutually inverse, which yields the desired bijection
                                $\mathcal{M}_{p,q}(\alpha,\beta)\cong RT(\alpha,\beta)$.
                        \end{proof}
                        
                        \begin{example}\label{lastex}
                                We illustrate here the map $T$ from the above proof. For instance, for $\alpha=(3,3,2,1,1)$ and $\beta=(2,2,2,2,1,1)$, we have:
                                \[
                                a=
                                \begin{pmatrix}
                                        0 & 1 & 1 & 1 & 0 & 0\ \\ 1 & 1 & 0 & 0 & 0 & 1 \\ 0 & 0 & 1 & 0 & 1 & 0 \\ 0 & 0 & 0 & 1 & 0 & 0 \\ 1 & 0 & 0 & 0 & 0 & 0
                                \end{pmatrix}
                                \in\mathcal{M}_{5,6}(\alpha,\beta)
                                \quad\Rightarrow\quad T(a)=\young(234,126,35,4,1)\in \mathrm{RT}(\alpha,\beta).
                                \]
                                
                        \end{example}
                        
                        \section*{Index of notation}
                        
                        \noindent
                        \S\ref{section-1.1}: $\ft_\mathbb{R}$, $\ft_\mathbb{R}^*$, $\mult_V(\mu)$, $W(\fh,\ft)$. \\
                        \S\ref{section-1.2}: $\fg$, $\ft$, $\fh$, $\fq$, $\fp^{\pm}$, $\Phi$, $\Phi(\fh)$, $\Phi(\fq)$, $\Phi^+$, $\Phi^+ (\fh)$, $\Phi^+ (\fq)$, $\rho(\fg)$, $\rho(\fh)$, $\rho(\fq)$, $W(\fg,\ft)$, $\spin$, $\Lambda(\spin)$, $\mu(A)$, $\mult_\spin(\beta)$. \\
                        \S\ref{section-2.1}: ${}^w \spin$, ${}^w \fh$, ${}^w \fq$, $W'$, $\fh(c)$. \\
                        \S\ref{section-2.2}: $\fp_I$, $\fl_I$, $\Phi_I$, $\Phi^+_I$, $\Phi_I(\fq)$, $\Phi^+_I(\fq)$, $\spin_I$, $\mu_I(A)$. \\
                        \S\ref{section-4}: $\epsilon_i$, $\alpha_i$, $\varpi_i$. \\
                        \S\ref{S4.1}: $\sigma$, $\mathcal{M}(\fq)$, $\mu_a$, $a_{i,*}$, $a_{*,j}$. \\
                        \S\ref{section-inductive}: $\mult_{(c_1,\ldots,c_k)}(b_1,\ldots,b_n)$, $\mathcal{P}_k(X)$, $\mathbf{1}_J$. \\
                        \S\ref{section-5.1}: $\mathrm{Part}(\binom{n}{2};n)$, $\mathcal{P}(n)$, $\lambda\preceq\lambda'$, $\lambda^0$, $\mu[\lambda]$, $\Lambda^+ (\spin)$. \\
                        \S\ref{section-4.3}: $N_\lambda$, $M_p(\lambda)$. \\
                        \S\ref{section-4.4}: $\mathcal{ST}(\lambda)$, $N_\tau$. \\
                        \S\ref{section6}: $\mu(\alpha,\beta)$, ${}^t \alpha$, $\Lambda'(\spin)$, $RT(\alpha,\beta)$.

                        %%% REFERENCES %%%
                        {\small\bibliography{spin}}
                        % Please, do not change the above line and do not insert your references
                        % into this file. Instead, insert your references into the commat.bib file.
                        % See commat.bib for further instructions.
                        
                        \EditInfo{April 5, 2023}{July 7, 2023}{Pasha Zusmanovich}

                \end{document}